\documentclass[12pt,reqno]{amsart}

\usepackage[normalem]{ulem}
\usepackage[T1]{fontenc}
\usepackage{enumitem}
\usepackage{hyperref}
\hypersetup{
  colorlinks   = true,
  citecolor    = blue,
  linkcolor    = blue
}

\usepackage{amsmath,amsfonts,amsthm,amssymb,color,tikz, comment}
\usepackage{mathrsfs}   
\usepackage{comment}
\usepackage{pdfsync}
\usepackage[font={scriptsize}]{caption}
\usepackage{ulem}

\usepackage[left=1in, right=1in, top=1.1in,bottom=1.1in]{geometry}
\def\N{{\mathbb N}}
\def\R{{\mathbb R}}

\def\eps{\varepsilon}
\def\E{{\mathbb E}}
\def\P{{\mathbb P}}
\def\Z{{\mathbb Z}}

\newtheorem{theorem}{Theorem}[section]

\newtheorem{conjecture}[theorem]{Conjecture}

\newtheorem{assumption}[theorem]{Assumption}
\newtheorem{lemma}[theorem]{Lemma}

\newtheorem{proposition}[theorem]{Proposition}

\theoremstyle{remark}

\theoremstyle{remark}

\begin{document}

\title[Polarity of points for Gaussian random fields]{Polarity of
points for Gaussian random fields in critical 
dimension}

\author[Y. Hakiki]{Youssef Hakiki}
\address{Department of Mathematics, Purdue University, West 
Lafayette, IN 47907, United States}
\email{yhakiki@purdue.edu}

\author[C. Y. Lee]{Cheuk Yin Lee}
\address{School of Science and Engineering, The Chinese University of Hong Kong (Shenzhen), Longgang, Shenzhen, Guangdong, 518172, China}
\email{leecheukyin@cuhk.edu.cn}

\author[Y. Xiao]{Yimin Xiao}
\address{Department of Statistics and Probability, Michigan 
State University, East Lansing, MI 48824, United States}
\email{xiaoy@msu.edu}



%\keywords{}
%
%\subjclass[2010]{}

\begin{abstract}
We study the property of hitting points for a class of $\R^d$-valued continuous Gaussian random fields on $\R^N$ with 
stationary increments, i.i.d.~coordinates, and a regularly 
varying variance function $\sigma$ of index $0<H<1$. We first 
prove that if
\[
 	\lim_{r\to 0^+} \frac{r^N}{\sigma^d\left(r\left( \log\log\frac{1}{r}\right)^{-1/N}\right)} = \infty,
\]
then every fixed point is polar (i.e., not hit almost surely). 
In general, this criterion may not be optimal in the critical dimension $d=N/H$. To aim for an optimal condition, we consider 
the specific case $\sigma(r) = r^H (\log(1/r))^\gamma$ and prove
that, in the critical dimension $d=N/H$, points are polar if and 
only if $\gamma \le 1/d$, or equivalently in this specific case,
\[
 	\int_{0^+} \frac{r^{N-1}}{\sigma^d(r)} dr = \infty.
\]
This integral condition is also necessary for points to be polar 
under general assumptions. Our main contribution lies in the 
proof of sufficiency of this condition in the specific case, 
where we extend a covering argument of Talagrand (1998) based on
sojourn time estimates to obtain Hausdorff measure bounds and 
solve polarity of points in the critical dimension.
% We study the polarity of points for a broad class of Gaussian random fields with stationary increments in the critical dimension. Our work extends previous results on the critical dimension—particularly those of Dalang, Mueller, and Xiao [Ann. Probab. 45 (2017), no. 6B, 4700–4751]—by considering Gaussian random fields whose variance functions are regularly varying, with a slowly varying component that may diverge at the origin. We prove that points are polar in the critical dimension under certain conditions on the variance function. Our approach extends the covering argument, based on optimal sojourn time estimates, developed by M. Talagrand [Probab. Theory Relat. Fields 112 (1998), 545–563] for multiparameter fractional Brownian motion. The results are applicable to a large class of Gaussian random fields.
\end{abstract}

\maketitle
% \tableofcontents

\section{Introduction}

Consider an $\R^d$-valued continuous Gaussian random field
$X = \{ X(t) , t \in \R^N\}$ on a complete probability space
$(\Omega,\mathscr{F}, \P)$ with $X(0)=0$.
For any compact set $A \subset \R^d$, we say that $A$ is 
\emph{polar} for $X$ if $\P\{ \exists\, t \in \R^N \setminus
\{0\}$ such that $ X(t) \in A \} = 0$. 
In particular, we say that \emph{points are polar} for $X$ if 
every fixed point in $\R^d$ is polar for $X$, i.e., 
\[
\P\{ 
\exists \, t \in \R^N \setminus\{0\} \text{ such that } X(t) = 
z\} = 0 \quad \forall z \in \R^d.
\]
It is an important and challenging problem in probabilistic
potential theory to determine the polarity of a given set $A$ for 
a Gaussian random field. Except for the seminal work of 
Khoshnevisan and Shi \cite{KS99} for the Brownian sheet, this problem has 
not 
been resolved completely for other Gaussian random fields
and has attracted a lot of attention in recent years. We 
refer to 
\cite{BLX09,EH25,HS,K85,LSXY23,LX26,Te86,X99,X09} for necessary conditions and 
sufficient conditions for $A$ to be polar for Gaussian random fields, and to 
\cite{DKN07,DKN09,DKN13,DN04,DS10, DS15} for related results for the solutions 
of systems of stochastic partial differential equations (SPDEs). 

In the special case when $A$ is a singleton, typically there is a critical 
value $d_c$, which is called the \emph{critical dimension}, such that points 
are polar if $d>d_c$ and non-polar if $d<d_c$. When $d=d_c$, it is 
usually more difficult to determine whether or not points are 
polar. 
For example, if $X=\{X(t), t \in \R^N \}$ is a $d$-dimensional fractional 
Brownian motion with Hurst index $H \in (0,1)$, then points are polar if $d 
>N/H$ and non-polar if $d < N/H$ (see \cite{K85, Te86}). 
Dalang, Mueller, and Xiao \cite{DMX17} proved 
that points are polar in the critical case $d=N/H$ by extending the random covering argument of Talagrand \cite{T95}.
%based on minimal oscillations as characterized by Chung's law of the iterated logarithm or small ball probabilities. 
%The argument of Talagrand \cite{T95} was also applied in \cite{X97} to study the exact Hausdorff measure of the level set $X^{-1}(z)$ when $d<N/H$.

The goal of this paper is to investigate the polarity of points, including the case of critical dimension, for a class of Gaussian random fields with stationary increments and a regularly varying variance function.

Let $X=\{X(t), t \in \R^N\}$ be an $\R^d$-valued continuous centered Gaussian random field defined on a complete probability space $(\Omega, \mathscr{F}, \P)$ that satisfies $X(0)=0$ and has stationary increments, i.i.d.~coordinates $X(t) = (X_1(t),\dots, X_d(t))$, and a continuous covariance function $R(s,t) := \E[X_1(s)X_1(t)]$.
It follows from Yaglom \cite{Yaglom1957, Yaglom1987} that $R$ can be written as
\begin{align*}
	R(s,t)=s' M t + \int_{\R^N} (e^{is\cdot\xi}-1)(e^{-it\cdot\xi}-1) m(d\xi)
\end{align*}
for some $N\times N$ nonnegative definite matrix $M$ and nonnegative symmetric measure $m$ on $\R^N\setminus\{0\}$ (called the \emph{spectral measure} of $X$) satisfying
\begin{align*}
\int_{\R^N} \frac{|\xi|^2}{1+|\xi|^2} m(d\xi) < \infty.
\end{align*}

We will make use of the following assumptions.
\begin{assumption}\label{a:X}
$M=0$, so that
\begin{align}\label{E:cov}
	R(s,t)=\int_{\R^N} (e^{is\cdot\xi}-1)(e^{-it\cdot\xi}-1) m(d\xi).
\end{align}
Moreover, there exist $\delta_0>0$ and a nondecreasing, continuous, regularly varying function $\sigma: [0, \delta_0] \to [0, \infty)$ with $\sigma(0)=0$ of the form
\begin{align}\label{sigma}
	\sigma(r) = r^H L(r), \quad r\in(0, \delta_0],
\end{align}
for some constant $H\in(0,1)$ and slowly varying function $L: (0,\delta_0] \to (0,\infty)$, and there exists a constant $0<c_1<\infty$ such that 
\begin{align}\label{variogram}
	d(s,t) \le c_1 \sigma(|s-t|)
\end{align}
uniformly for all $s, t\in \R^N$ with $|t-s|\le \delta_0$, where $d$ is the canonical metric defined by $d(s,t) = (\E[(X_1(s)-X_1(t))^2])^{1/2}$.
\end{assumption}

\begin{assumption}\label{a:var}
    $\mathrm{Var}(X_1(t)) > 0$ for all $t \in \R^N \setminus\{0\}$.
\end{assumption}

\begin{assumption}\label{a:SLND}
There exists a constant $0<c_2<\infty$ such that
\begin{align}\label{SLND}
	\mathrm{Var}(X_1(t)|X_1(s) : r \le |s-t|\le \delta_0) \ge c_2 \sigma^2(r)
\end{align}
uniformly for all $t\in \R^N$ and $r\in (0, |t|\wedge \delta_0]$.
\end{assumption}

Note that \eqref{variogram} and \eqref{SLND} together imply $d(s,t)=\E[(X_1(t)-X_1(s))^2]\big) \asymp \sigma^2(|s-t|)$ for all $s, t \in \R^N$ with $|s-t|\le \delta_0$.
Gaussian random fields satisfying \eqref{variogram} in Assumption \ref{a:X} are called \emph{approximately isotropic} and those satisfying Assumption \ref{a:SLND} are called \emph{strongly locally $\sigma$-nondeterministic} (see \cite{X96}). 
The class of Gaussian random fields satisfying Assumptions \ref{a:X} and \ref{a:SLND} is large. It includes fractional Brownian motion, the fractional Riesz-Bessel motion \cite{Anh99, X07}, the solution $\{u(t, x), t \ge 0, x \in \R^d\}$ to SPDE with the generator of a L\'evy process and 
additive fractional-colored Gaussian noise (viewed as a random field in the 
space variable $x$, when $t>0$ is fixed) \cite{FKN11, HSWX}, and the examples given in \cite[Chapter 7]{MarcusRosen06}. 
For these Gaussian random fields, Assumption \ref{a:X} can be verified by using a stochastic 
integral representation (see \eqref{E:spectral2} below) and the asymptotic properties of the
spectral measure at infinity (cf. \cite[Theorem 1]{Pitman68} or more 
generally, \cite[Theorem 2.5]{X07}); Assumption \ref{a:SLND} can be verified by using the Fourier-analytic method for 
proving the strong local nondeterminism in \cite{Pitt78, X07}. Moreover, 
given any regularly varying function $\sigma$ of the form \eqref{sigma} such 
that the slowly varying function $L(r)$ is eventually monotone, then by 
\cite[Theorem 5]{Pitman68} and a Tauberian theorem, we can find a corresponding Gaussian random field satisfying Assumptions \ref{a:X} and \ref{a:SLND}. We refer to \cite{X07} for more information. 

%, and the examples given in \cite[Chapter 7]{MarcusRosen06}. 
%Conditions (i) and (ii) for these Gaussian random fields can be established by using the stochastic integral representation \eqref{E:spectral2}, the asymptotic properties of the spectral measure $m$ at infinity, \cite[Theorem 1]{Pitman68} or more generally, \cite[Theorem 2.5]{X07}, and the Fourier-analytic method for 
%proving the strong local nondeterminism in \cite{Pitt78, X07}. 
%Moreover, given any regularly varying function $\sigma$ of the form (\ref{sigma}) such 
%that the slowly varying function $L(r)$ is eventually monotone, then by 
%\cite[Theorem 5]{Pitman68} and a Tauberian theorem, we can find a spectral 
%measure $m$ such that the corresponding Gaussian random field satisfies 
%Conditions (i) and (ii). 
%We refer to \cite{X07} for more information. 

Our first theorem provides a sufficient condition for points to be polar 
for $X$. 
%One can check that Condition \eqref{suff-cond} below implies \eqref{Eq:Conj}.

\begin{theorem}\label{T:polar:suff}
Let Assumptions \ref{a:X} and \ref{a:var} hold. If
\begin{align}\label{suff-cond}
\lim_{r\to 0^+} \frac{r^N}{\sigma^d\left(r\left( \log\log\frac{1}{r}\right)^{-1/N}\right)} = \infty,
\end{align}
then points are polar for $X$.
\end{theorem}

Theorem \ref{T:polar:suff} is proved by an extension of the covering argument in \cite{DMX17,T95,X97} which is based 
on small oscillations characterized by Chung's law of the 
iterated logarithm or small ball probabilities (see Section \ref{S:pf:T1}).
As a consequence of Theorem \ref{T:polar:suff}, if $d>N/H$, then points are polar.
However, in the critical dimension $d=N/H$, the criterion \eqref{suff-cond} does not give an optimal condition for the polarity of points.
To illustrate this, suppose Assumptions \ref{a:X}--\ref{a:SLND} hold with
\begin{align}\label{sigma:gamma}
	\sigma(r) = r^H \left( \log \frac{1}{r} \right)^\gamma, \quad \text{where $H \in (0,1)$ and $\gamma \in \R$.}
\end{align}
If $d=N/H$ and $\gamma \le 0$, then \eqref{suff-cond} holds, hence points are polar by Theorem \ref{T:polar:suff}; if $d=N/H$ and $\gamma > 0$, then \eqref{suff-cond} does not hold:
\begin{align*}
\lim_{r\to 0^+} \frac{r^N}{\sigma^d\left(r\left( \log\log\frac{1}{r}\right)^{-1/N}\right)} = \lim_{r\to 0}\frac{\log\log\frac{1}{r}}{\left(\frac{1}{N} \log\log\log\frac{1}{r} + \log\frac{1}{r} \right)^{\gamma d}} = 0.
\end{align*}
However, it is known \cite{Berman69, Pitt78, GH80} that, under Assumptions \ref{a:X}--\ref{a:SLND}, $X$ has a square-integrable local time on any interval $I \subset \R^N$ if and only if
\begin{equation}\label{Eq:LT}
\int_I \int_I \frac{dt\, ds}{\big[\E(X_1(s) - X_1(t))^2\big]^{d/2}} < \infty,
\end{equation}
which, under the additional assumption of \eqref{sigma:gamma} and $d=N/H$, is equivalent to $\gamma d \le 1$.
Hence, in this case, we conjecture that points are polar if and only if $\gamma \le 1/d$.

In order to verify this conjecture, we replace the covering argument based on small oscillations by a new covering argument, which is an extension of Talagrand's covering argument based on sojourn time 
estimates \cite{T98}, which is more effective in 
the critical case $d=N/H$ than that in \cite{DMX17} and is the main 
contribution of the present paper. 
In particular, this new covering argument yields a more precise bound for the Hausdorff measure of the range $X(I)$.
Recall that the \emph{Hausdorff measure} of a set $A \subset \R^d$ with respect to a gauge function $\phi(r)$ is defined by
\[
    \mathcal{H}^\phi(A) = \lim_{\delta \to 0^+} \inf\left\{ \sum_{n=1}^\infty \phi(\operatorname{diam}{U_n}) : \bigcup_{n=1}^\infty U_n \supset A \ \text{ with }\ \sup_{n} \operatorname{diam}{U_n} \le \delta \right\},
\]
where $\operatorname{diam}$ denotes diameter; see \cite{Falconer, Mattila} for more information.

\begin{theorem}\label{T:Haus:meas}
Let Assumptions \ref{a:X} and \ref{a:SLND} hold with $\sigma$ given by \eqref{sigma:gamma}, where $d=N/H$ and $\gamma \le 1/d$. Let $I$ be a compact interval in $\R^N$. Then a.s., the range $X(I)$ has finite Hausdorff measure with respect to the gauge function $\phi(r)$ defined by
\begin{align}\label{Eq:phi_def}
    \phi(r) =  r^d (\log(1/r))^{1-\gamma d}\log\log\log(1/r).
\end{align}
In particular, $X(I)$ has Lebesgue measure zero a.s.
\end{theorem}

The conclusion of Theorem \ref{T:Haus:meas} is key to obtaining an optimal 
condition for points to be polar under \eqref{sigma:gamma}, which we state as part of the following result.

\begin{theorem}\label{T:polar}
Under Assumptions \ref{a:X}, \ref{a:var}, and \ref{a:SLND},
% Points are polar for $X$ if and only if \eqref{intsigma} holds.
the following statements hold:
\begin{enumerate}
    \item[(i)] If there is a fixed point $z\in \R^d$ that is polar for $X$, then
\begin{equation}\label{E:int:cond}
 \int_0^{\delta_0} \frac {r^{N-1}}{\sigma^d(r)}\, dr  = \infty.
\end{equation}
    \item[(ii)] Under \eqref{sigma:gamma}, condition \eqref{E:int:cond} implies that points are polar for $X$.
\end{enumerate}
\end{theorem}

Note that, under \eqref{sigma:gamma},
\begin{align}
	\eqref{E:int:cond} \Leftrightarrow \begin{cases}
	\text{$d>N/H$, or}\\
	\text{$d=N/H$ and $\gamma \le 1/d$.}
	\end{cases}
    \label{int:cond:under:gamma}
\end{align}
Hence, \eqref{int:cond:under:gamma} is a necessary and sufficient condition for points to be polar under Assumptions \ref{a:X}--\ref{a:SLND} and \eqref{sigma:gamma}, thereby verifying our earlier conjecture.
Furthermore, note that the above condition \eqref{Eq:LT} for existence of local times is equivalent to the integral in \eqref{E:int:cond} being finite.
%This leads to a natural question: What is a general criterion for determining whether points are polar for $X$?
In general, the polarity of points for a stochastic process, say $Y$, is closely related to the non-existence of local times of $Y$. For L\'evy processes, it follows from the seminal works of Kesten \cite{Kesten69} and Hawkes \cite{H86} that the polarity of points is equivalent to the non-existence of local times. This equivalence remains valid for additive L\'evy processes (\cite{KXZa, KXZb}) and for fractional Brownian motion (\cite{DMX17, GH80, Pitt78}). 
In light of these facts and Theorem \ref{T:polar}, we believe that this equivalence also holds for any Gaussian random field $X$ that satisfies Assumptions \ref{a:X}--\ref{a:SLND}. 
%\eqref{vgram-iso}.
%Since, in this case, $X$ has a square-integrable local time on any interval $I \subset \R^N$ (e.g., $I = [0, 1]^N$) if and only if
%\begin{equation}\label{Eq:LT}
%\int_I \int_I \frac{dt\, ds}{\big[\E(X_1(s) - X_1(t))^2\big]^{d/2}} < \infty,
%\end{equation}
%(see \cite{Berman69, Pitt78, GH80}), which, under Assumption (A), is equivalent to 
%\begin{equation}\label{Eq:LT2}
% \int_0^{\delta_0} \frac {r^{N-1}}{\sigma^d(r)}\, dr < \infty,
%\end{equation}
%we have the following conjecture:
\begin{conjecture}\label{Con}
Under Assumptions \ref{a:X}, \ref{a:var} and \ref{a:SLND}, points 
are polar for $X$ if and only if \eqref{E:int:cond} holds.
\end{conjecture}

The rest of this paper is organized as follows. 
% In Section 2, we provide some preliminaries and assumptions, and state our main results, Theorems 2.1, 2.2 and 2.3. 
In Section \ref{S:pf:T1}, we prove Theorem \ref{T:polar:suff}.
In Section \ref{S:sojourn}, we establish sharp sojourn time estimates under \eqref{sigma:gamma} in the critical dimension $d=N/H$.
In Sections \ref{S:pf:T2} and \ref{S:pf:T3}, we prove Theorems \ref{T:Haus:meas} and \ref{T:polar}, respectively.
% {\textcolor{blue}{ We will add a little more description in this part.} }

Throughout the paper,
$B(t,r)$ denotes the closed ball centered at $t$ with radius $r$; $\lambda_N$ denotes Lebesgue measure on $\R^N$; $a \wedge b = \min\{a,b\}$; $\log$ denotes natural logarithm; $\log_2$ denotes base-2 logarithm; ``$f(x) \lesssim g(x)$'' means that there exists a constant $C$ such that $f(x) \le C g(x)$ for all $x$; ``$f(x) \asymp g(x)$'' means that $f(x) \lesssim g(x)$ and $g(x) \lesssim f(x)$; ``$f(x)\sim g(x)$ as $x\to0$'' means that $\lim_{x\to0}f(x)/g(x) = 1$.

\section{Proof of Theorem \ref{T:polar:suff}}
\label{S:pf:T1}

% \begin{lemma}\cite[Proposition 3.1]{X96}\label{L:P:Chung}
%     There exist constants $r_1 > 0$ and $K_1<\infty$ such that for all $r_0 \in (0,r_1]$,
%     \begin{align*}
%         \P\left\{ \exists\, r \in [r_0^2,r_0],\, \sup_{|t|\le r}|X(t)| \le K_1 \sigma\left(r\left( \log\log \tfrac{1}{r} \right)^{-1/N}\right) \right\} \ge 1 - \exp\left(-\left( \log\tfrac{1}{r_0} \right)^{1/2}\right).
%     \end{align*}
% \end{lemma}

\begin{proof}
    Fix $z \in \R^d$. We will prove that $\{z\}$ is polar for $X$ by first applying the covering argument of \cite{T95, X97} for estimating Hausdorff measures, and then extending the method of \cite{DMX17} to conclude that $\{z\}$ is polar.
    Let us recall and follow the set-up in \cite{T95, X97} for the covering argument.
    Let $t_0 \in \R^N \setminus\{0\}$.
    Define the Gaussian random fields $X^{(1)} = \{X^{(1)}(t), t \in \R^N\}$ and $X^{(2)} = \{X^{(2)}(t), t \in \R^N\}$ by
    \begin{align}\label{X2:X1}
        X^{(2)}(t) = \E[X(t)|X(t_0)], \quad X^{(1)}(t) = X(t) - X^{(2)}(t).
    \end{align}
    Then $X^{(1)}$ and $X^{(2)}$ are independent.
    Also, $X^{(1)}$ is independent of $X(t_0)$.
    In particular, for each $j \in \{1,\dots, d\}$,
    \begin{align}\label{X2}
        X_j^{(2)}(t) = \alpha(t) X_j(t_0), \quad \text{where} \quad
        \alpha(t) = \frac{\E[X_j(t)X_j(t_0)]}{\E[(X_j(t_0))^2]}
    \end{align}
    and $\alpha(t)$ does not depend on $j$.
    By Assumption \ref{a:var} and continuity, we may choose a small number $\rho_0 \in (0,\delta_0)$ such that
    \begin{align}\label{alpha:UB:LB}
        1/2 \le \alpha(t) \le 3/2 \quad \text{for all $t \in I$,}
    \end{align}
    where $I$ is a closed interval centered at $t_0$ with diameter $\rho_0$ and $I \subset \R^N \setminus\{0\}$.
    As was pointed out in \cite{X97}, we may assume, according to Theorem 1.8.2 of \cite{BGT89}, that $L: (0,\delta_0) \to (0,\infty)$ is smooth.
    Then, by Lemma 4.2 of \cite{X97}, there exists a constant $K_0<\infty$ such that
    \begin{align}\label{E:alpha-alpha}
        |\alpha(t)-\alpha(s)| \le K_0 |t-s|^{\gamma} \quad \text{for all $t,s \in I$,}
    \end{align}
    where $\gamma = 2H \wedge 1$, and hence
    \begin{align}\label{E:X2-X2}
        |X^{(2)}(t)-X^{(2)}(s)| \le K_0 |t-s|^{\gamma} |X(t_0)| \quad \text{for all $t,s \in I$.}
    \end{align}
    Choose and fix $\beta >0$ such that $H+\beta < \gamma$.
    For $p \ge 1$, consider the random sets 
    \begin{align*}
        &R_p = \left\{ t \in I : \exists \, r \in [2^{-2p}, 2^{-p}], \, \sup_{s \in B(t,r)} |X(s)-X(t)| \le K_1 \sigma\left(r\left( \log\log \tfrac{1}{r} \right)^{-1/N}\right)\right\},\\
        &R_p' = \left\{t \in I : \exists \, r \in [2^{-2p}, 2^{-p}], \, \sup_{s \in B(t,r)} |X^{(1)}(s)-X^{(1)}(t)| \le K_2 \sigma\left(r\left( \log\log \tfrac{1}{r} \right)^{-1/N}\right) \right\},
    \end{align*}
    where the constants $K_1$ and $K_2$ will be chosen below, and the events
    \begin{align*}
        &\Omega_{p,1} = \left\{ \lambda_N(R_p) \ge \lambda_N(I) \left(1-e^{-\sqrt{p}/4}\right) \right\},\\
        &\Omega_{p,2} = \left\{ \sup_{t \in I}|X(t)| \le 2^{\beta p} \right\},\\
        &\Omega_{p,3} = \left\{ \lambda_N(R_p') \ge \lambda_N(I) \left(1-e^{-\sqrt{p}/4}\right) \right\},\\
        &\Omega_{p,4} = \left\{ \sup_{t,s\in I: |t-s|\le \sqrt{N}2^{-p}}|X(t)-X(s)| \le K_3 \sigma(2^{-p}) \sqrt{p} \right\}.
    \end{align*}
    As was proved in \cite[p. 152]{X97}, there exists a constant $K_1>0$ such that
    the probabilities of the complement of $\Omega_{p,1}$ ($p \ge 1$) are summable, i.e.,
    \begin{align*}
        \sum_{p=1}^\infty \P\{\Omega_{p,1}^c\} < \infty.
    \end{align*}
    %where $E^c$ denotes the complement of an event $E$. 
    By Dudley's theorem \cite{Dudley}, $\E[\sup_{t\in I}|X(t)|]<\infty$. This and Markov's inequality imply that
    \begin{align*}
        \sum_{p=1}^\infty \P\{ \Omega_{p,2}^c\} \le \sum_{p=1}^\infty 2^{-\beta p} \,\textstyle{\E[\sup_{t \in I}|X(t)|]} < \infty.
    \end{align*}
    By \eqref{X2:X1} and \eqref{E:X2-X2}, if $t \in R_p$ and $\Omega_{p,1} \cap \Omega_{p,2}$ occurs, then there exists $r \in [2^{-2p},2^{-p}]$ such that for all $s \in B(t,r)$,
    \begin{align*}
        |X^{(1)}(t)-X^{(1)}(s)| 
        &\le |X(t)-X(s)| + |X^{(2)}(t)-X^{(2)}(s)|\\
        & \le K_1 \sigma\left(r\left( \log\log \tfrac{1}{r} \right)^{-1/N}\right) + K_0 r^\gamma 2^{\beta p}\\
        & \le K_2 \sigma\left(r\left( \log\log \tfrac{1}{r} \right)^{-1/N}\right)
    \end{align*}
    for some constant $K_2>0$ (since $\gamma-\beta>H$). 
    This shows that $K_2$ can be chosen such that $R_p\subset R_p'$ on $\Omega_{p,1} \cap \Omega_{p,2}$, hence $\Omega_{p,1} \cap \Omega_{p,2} \subset \Omega_{p,3}$. It follows that
    \begin{align*}
        \sum_{p=1}^\infty \P\{\Omega_{p,3}^c\} \le \sum_{p=1}^\infty 
        \left(\P\{\Omega_{p,1}^c\} + \P\{\Omega_{p,2}^c\}\right) < \infty.
    \end{align*}
    By Lemma 3.1 of \cite{X96} and the stationarity of increments, there exists a constant $K_3>0$ such that
    \begin{align*}
        \sum_{p=1}^\infty \P\{\Omega_{p,4}^c\} < \infty.
    \end{align*}
    Now let $\Omega_p = \Omega_{p,3} \cap \Omega_{p,4}$, so that 
    \begin{align}\label{Op:io}
        \sum_{p=1}^\infty \P\{\Omega_p^c\} < \infty.
    \end{align}
    We say that $A$ is a dyadic cube in $\R^N$ of order $q$ if it is of the form $A = \prod_{j=1}^N [k_j2^{-q}, (k_j+1)2^{-q}]$ for some $k=(k_1, \dots, k_N) \in \Z^N$.
    For each dyadic cube $A$, let $t_A$ denote the center of $A$.
    As in the proof of Theorem 4.2 in \cite{X97}, for each $p \ge 1$, we can obtain a family of dyadic cubes $\mathscr{H}_p = \mathscr{H}_{1,p} \cup \mathscr{H}_{2,p}$, where $\mathscr{H}_{1,p}$ consists of non-overlapping dyadic cubes of order $q$, where $p \le q \le 2p$, that intersect $I$ and whose union contains $R_p'$, and $\mathscr{H}_{2,p}$ consists of non-overlapping dyadic cubes of order $2p$ that intersect $I$ but not contained in any cubes in $\mathscr{H}_{1,p}$. 
    Note that the cubes in $\mathscr{H}_p$ form a cover for the interval $I$. 

    Next, we employ this covering and extend the method of \cite{DMX17} to show that $\{z\}$ is polar. 
    Define
    \begin{align}\label{X3}
        X^{(3)}(t) = \frac{1}{\alpha(t)}(z-X^{(1)}(t)),
    \end{align}
    where $\alpha(t)$ is given in \eqref{X2}.
    We claim that the range $X^{(3)}(I)$ has Lebesgue measure 0.
    In fact, for any $A \in \mathscr{H}_p$ and $t \in A$, we have
    \begin{align}\begin{split}\label{X3-X3}
        X^{(3)}(t) - X^{(3)}(t_A) &= \frac{1}{\alpha(t)}(z-X^{(1)}(t)) - \frac{1}{\alpha(t_A)} (z-X^{(1)}(t_A))\\
        &= \frac{(\alpha(t_A)-\alpha(t))(z-X^{(1)}(t))+\alpha(t)(X^{(1)}(t_A) - X^{(1)}(t))}{\alpha(t)\alpha(t_A)}.
    \end{split}\end{align}
    If $A \in \mathscr{H}_{1,p}$ and $A$ is of order $q$, then by \eqref{X2:X1}--\eqref{E:alpha-alpha}, when $\Omega_p$ occurs,
    \begin{align*}
        |X^{(3)}(t) - X^{(3)}(t_A)|
        &\le 4 \left[ K_0 |t_A-t|^\gamma (|z| + (1+\tfrac{3}{2}) 2^{\beta p} ) + \tfrac{3}{2} K_2 \sigma(2^{-q}(\log\log 2^q)^{-1/N}) \right]\\
        & \le C_1 \sigma(2^{-q}(\log\log 2^q)^{-1/N})
    \end{align*}
    for some constant $C_1$, where we have used $|t_A-t|\le \sqrt{N}2^{-q}$ and $\gamma-\beta>H$ to obtain the last inequality.
    Similarly, if $A \in \mathscr{H}_{2,p}$, then when $\Omega_p$ occurs,
    \begin{align*}
        |X^{(3)}(t) - X^{(3)}(t_A)|
        &\le 4 \left[ K_0|t_A-t|^\gamma (|z|+(1+\tfrac{3}{2})2^{\beta p}) + \tfrac{3}{2} (1+\tfrac{3}{2}) K_3 \sigma(2^{-2p}) \sqrt{2p} \right]\\
        &\le C_2 \sigma(2^{-2p})\sqrt{p}
    \end{align*}
    for some constant $C_2$.
    This shows that for each $p \ge 1$, the family $\{ B(X^{(3)}(t_A), r_A) : A \in \mathscr{H}_p \}$ of balls form a random cover for $X^{(3)}(I)$ when the event $\Omega_p$ occurs, where, for each dyadic cube $A$,
    \begin{align}
        r_A = \begin{cases}
        C_1 \sigma(2^{-q}(\log\log 2^q)^{-1/N}) & \text{if $A \in \mathscr{H}_{1,p}$ is of order $q$,}\\
        C_2\sigma(2^{-2p})\sqrt{p} & \text{if $A \in \mathscr{H}_{2,p}$.}
        \end{cases}
    \end{align}
    But by \eqref{Op:io}, with probability 1, $\Omega_p$ occurs for all sufficiently large $p$, and when $\Omega_{p,3}$ occurs, $\# \mathscr{H}_{2,p} \le C_3 2^{2pN} e^{-\sqrt{p}/4}$ (since $\mathscr{H}_{2,p}$ covers $I\setminus R_p'$), hence we have
    \begin{align*}
        &\lambda_d(X^{(3)}(I)) \lesssim \sum_{A \in \mathscr{H}_p} r_A^d 
        = \sum_{A \in \mathscr{H}_{1,p}}  \left[C_1\sigma(2^{-q}(\log\log 2^q)^{-1/N})\right]^d + \sum_{A \in \mathscr{H}_{2,p}} \left[C_2 \sigma(2^{-2p})\sqrt{p} \right]^d\\
        &\le \max_{p \le q \le 2p}\frac{[C_1\sigma(2^{-q}(\log\log 2^q)^{-1/N})]^d}{2^{-qN}} \sum_{\substack{A \in \mathscr{H}_{1,p}\\\text{order }q}} 2^{-qN} + C_3 2^{2pN} \exp(-\sqrt{p}/4)\left[C_2 \sigma(2^{-2p})\sqrt{p} \right]^d.
    \end{align*}
    To estimate the last summand, we notice that, since $\sigma$ is regularly varying with index $H$, Theorem 1.5.6 of \cite{BGT89} shows that given $\delta>0$, there exists $p_0 \ge 1$ such that
    \begin{align*}
        \frac{\sigma(2^{-2p})}{\sigma(2^{-2p}(\log\log 2^{2p}))^{-1/N})} \le (\log\log2^{2p})^{\frac{H+\delta}{N}} \quad \text{for all $p \ge p_0$.}
    \end{align*}
    Also, recall that $\mathscr{H}_{1,p}$ consists of non-overlapping cubes that cover $R_p \subset I$.
    These together with condition \eqref{suff-cond} imply that with probability 1, for $p$ large,
    \begin{align*}
        \lambda_d(X^{(3)}(I)) 
        & \lesssim \max_{p \le q \le 2p}\frac{\sigma^d(2^{-q}(\log\log 2^q)^{-1/N})}{2^{-qN}} \sum_{A \in \mathscr{H}_{1,p}}\lambda_N(A)\\
        & \qquad + \exp(-\sqrt{p}/4)\,\frac{\sigma^d(2^{-2p}(\log\log 2^{2p})^{-1/N})}{2^{-2pN}}\,  (\log\log2^{2p})^{d\frac{H+\delta}{N}}\,p^{d/2} \\
        & \lesssim \max_{p \le q \le 2p}\frac{\sigma^d(2^{-q}(\log\log 2^q)^{-1/N})}{2^{-qN}}\bigl( \lambda_N(I) + o(1)\bigl)\\
        & \to 0 \quad \text{as $p \to \infty$.}
    \end{align*}
    We have thus verified the claim that $X^{(3)}(I)$ has Lebesgue measure 0 a.s.
    
    Finally, thanks to Fubini's theorem and the above claim, we have
    \begin{align*}
        \int_{\R^d} \P\{\exists\, t \in I, X^{(3)}(t) = x \} dx = \E[\lambda_d(X^{(3)}(I))] = 0,
    \end{align*}
    which implies that
    \begin{align}\label{X3:polar}
        \P\{\exists\, t \in I, X^{(3)}(t) = x\} = 0 \quad \text{for almost every $x \in \R^d$.}
    \end{align}
    Recall that $X^{(1)}$ is independent of $X(t_0)$.
    So, according to \eqref{X3}, $X^{(3)}$ is independent of $X(t_0)$.
    Also, by \eqref{X2},
    \begin{align}\label{X:X3}
        X(t) = z \quad \text{if and only if} \quad X^{(3)}(t) = X(t_0).
    \end{align}
    Let $f_0(x)$ be the probability density function of $X(t_0)$. 
    Thanks to \eqref{X:X3}, independence, and \eqref{X3:polar}, we deduce that
    \begin{align*}
        \P\{ \exists\, t \in I, X(t) = z \}
        &= \P\{ \exists\, t \in I, X^{(3)}(t) = X(t_0) \}\\
        &= \int_{\R^d} \P\{ \exists\, t \in I, X^{(3)}(t) = x \} f_0(x) dx = 0.
    \end{align*}
    Recall that $I \subset \R^N \setminus\{0\}$ is a closed interval centered at $t_0$ with diameter $\rho_0>0$.
    Since we can cover $\R^N \setminus \{0\}$ by countably many such intervals, it follows that
    \[
        \P\{ \exists \, t \in \R^N \setminus\{0\}, X(t)=z \} = 0.
    \]
    Hence, $\{z\}$ is polar for $X$.
    This completes the proof of Theorem \ref{T:polar:suff}.
\end{proof}

\section{Sojourn time estimates in the critical dimension}\label{S:sojourn}

This section aims to prove sharp estimates for the moments and tail probabilities of the ``truncated'' sojourn time
\begin{align*}
% \label{Def:sojourn}
T_{\eps} := \lambda_N\{ t \in \R^N : |t| \le \varepsilon^{\beta}, |X(t)| \le \eps\} = \int_{B_{\varepsilon^{\beta}}} {\bf 1}_{\{|X(t)|\le\eps\}} dt,
\end{align*}
for a fixed $\beta\in (1,1/H)$, where
\begin{align*}
% \label{E:B_0}
B_{\varepsilon^{\beta}}:=\{t\in \R^N : |t|\le \varepsilon^{\beta}\}.
\end{align*}
Throughout this section, we let Assumptions \ref{a:X} and \ref{a:SLND} hold with $\sigma$ given by \eqref{sigma:gamma}, i.e.,  
\begin{align}\label{E:sigma-def}
\sigma(r) = r^{H} L(r), \ \ \hbox{ where }\ L(r) = \left(\log(1/r)\right)^{\gamma},
\end{align}
%I changed $\asymp$ back to $=$ since we already assume the canonical metric $d(s,t) \asymp \sigma(|t-s|)$; see (3) and (4) in Assumptions 1.1 and 1.3
with $N=Hd$ and $\gamma \le 1/d$. 
An asymptotic inverse of $\sigma$ is given by
\begin{align}\label{D:sigma*}
    \sigma^*(r) := H^{\gamma/H} r^{1/H} (\log(1/r))^{-\gamma/H}
\end{align}
so that
\[
\sigma(\sigma^*(r)) \sim r \quad \text{and} \quad \sigma^*(\sigma(r)) \sim r 
\quad \text{as $r \to 0^+$.}
\]
% The associated function $f$ has the following form
Define
\begin{align}\label{D:f_cases}
f(r) := \begin{cases}
\left(\log(1/r)\right)^{1-\gamma d} & \text{if } \gamma < 1/d, \\
\log\log(1/r) & \text{if } \gamma = 1/d.
\end{cases}
\end{align}
We establish the following upper bounds for the moments of $T_\eps$.

\begin{lemma}\label{L:sojourn:UB}
There exist constants $C_1 < \infty$ and $\delta_1 \in (0, 1)$ such that 
for all integers $n \ge 1$ and all $\eps \in (0, \delta_1)$,
\begin{align}\label{E:sojourn:UB-Psi}
\E[T_\eps^n] \le \left[C_1 n \eps^d \Psi(\varepsilon)\right]^n,
\end{align}
where
    \begin{align}\label{E:psi}
    \Psi(\varepsilon) := (\log(1/\eps))^{1-\gamma d} =
    \begin{cases} 
    f(\varepsilon) & \text{if } \gamma < 1/d, \\
    1 & \text{if } \gamma = 1/d.
    \end{cases} 
    \end{align}
\end{lemma}

\begin{proof}
For any $n \ge 1$,
\begin{align}\label{E:ETn2a}
\E[T_\eps^n] = \int_{B_{\varepsilon^{\beta}}^n} \P\left\{|X(t_1)| \le \eps, \dots, |X(t_n)| \le \eps\right\} dt_1 \cdots dt_n.
\end{align}
Since the set of points $(t_1, \dots, t_n)\in B_{\varepsilon^{\beta}}^n$ such that $t_i = t_j$ 
for some $i \ne j$ has $(nN)$-dimensional Lebesgue measure $0$, the integration in \eqref{E:ETn2a} is effectively taken over the subset of $B_{\varepsilon^{\beta}}^n$ where all the points $t_1, \dots, t_n$ are distinct. 
By conditioning, we can write the above as
\begin{align}\begin{split}\label{E:ETn2}
\E[T_\eps^n] &=\int_{B_{\varepsilon^{\beta}}^{n-1}}dt_1\cdots dt_{n-1}\\
& \quad \times \int_{B_{\varepsilon^{\beta}}} dt_n \, \E\left[ {\bf 1}_{\{|X(t_1)|\le \eps, \dots, |X(t_{n-1})|\le \eps\}} \P\left\{ |X(t_n)| \le \eps | X(t_1),\dots, X(t_{n-1})\right\} \right].
\end{split}\end{align}
If we fix $t_1, \dots, t_{n-1}\in B_{\varepsilon^{\beta}}$, then for any $t_n \in B_{\varepsilon^{\beta}}$, the 
conditional distribution of $X_1(t_n)$ given $X_1(t_1),\dots, X_1(t_{n-1})$ 
is Gaussian. By the assumption of strong local nondeterminism (Assumption \ref{a:SLND}), 
the conditional variance of this distribution is bounded below as follows:
\begin{align*}
\mathrm{Var}(X_1(t_n)|X_1(t_1), \dots, X_1(t_{n-1})) \ge c_2 \,\sigma^2(r_n),
\end{align*}
where $r_n = \min\{ |t_n|, \min_{1\le i \le n-1}|t_n-t_i| \}$.
This together with Anderson's inequality \cite{A55} and the hypothesis that $X_1,\dots, X_d$ are i.i.d.~implies that
\begin{align*}
\P\{ |X(t_n)| \le \eps | X(t_1), \dots, X(t_{n-1}) \}
\le \P\left\{ |Z| \le \frac{\eps}{\sqrt{c_2}\,\sigma(r_n)} \right\}^d
\le \left(\min\left\{ 1,  \frac{2 \eps}{\sqrt{c_2}\,\sigma(r_n)} \right\}\right)^d,
\end{align*}
where $Z$ denotes a standard normal random variable.
Set $t_0=0$.
By simple estimates, the use of polar coordinates, and the relation $N = 
Hd$, we deduce that
\begin{align*}
\int_{B_{\varepsilon^{\beta}}} \left(\min\left\{ 1,  \frac{2 \eps}{\sqrt{c_2}\,\sigma(r_n)} \right\}\right)^d dt_n
&\le \sum_{i=0}^{n-1} \int_{B_{\varepsilon^{\beta}}} \min\left\{ 1, \frac{(2c_2^{-1/2})^d\,\eps^d}{\sigma^d(|t_n-t_i|)} \right\} dt_n\\
& \le C \sum_{i=0}^{n-1} \int_0^{\varepsilon^{\beta}} \min\left\{ 1, \frac{\eps^d}{\sigma^d(\rho)} \right\} \rho^{N-1} d\rho\\
& \le Cn \left[ \int_0^{\sigma^*(\eps)} r^{N-1} dr + \eps^d \int_{\sigma^*(\eps)}^{\varepsilon^{\beta}} \frac{d\rho}{\rho L^d(\rho)}  \right]\\
& \le Cn \left[ (\sigma^*(\eps))^N + \eps^d\bigl( f(\sigma^*(\eps))-f(\varepsilon^{\beta})\bigl) \right],
\end{align*}
valid uniformly for all distinct $t_1,\dots, t_{n-1}\in B_{\varepsilon^{\beta}}$.
Note that the above estimate is still valid when $n=1$ for which the 
conditional probability is replaced by the unconditional one.

We now analyze the term $A(\eps) := \eps^d( f(\sigma^*(\eps))-f(\varepsilon^{\beta}))$ for the two cases. Recall that $\sigma^*(\varepsilon) = C\varepsilon^{1/H}\left(\log(1/\varepsilon)\right)^{-\gamma/H}$ for all $\gamma\le 1/d$.
%and $\varepsilon^{\beta} \ll \varepsilon^{1/H}$ for small $\varepsilon$ since $\beta < 1/H$.

\medskip
\noindent\textbf{Case (1)} $\gamma < 1/d$:
In this case, $f(r) = (\log(1/r))^{1-\gamma d}$ and
\begin{align*}
A(\eps) \sim C\,\varepsilon^d\, \left[ \left(\log(1/\sigma^*(\eps))\right)^{1-\gamma d} - \left(\log(1/\eps^\beta)\right)^{1-\gamma d} \right] \quad \text{as $\eps \to 0$.}
\end{align*}
Using the expression of $\sigma^{*}$ it is not hard to show that $A(\eps) \sim C \eps^d f(\eps)$. In addition, this last term dominates $(\sigma^*(\eps))^N  = \eps^{N/H} \left(\log(1/\varepsilon)\right)^{-\gamma N/H}=\eps^{d}\left(\log(1/\varepsilon)\right)^{-\gamma d}$. Therefore,
\begin{align*}
\int_{B_{\varepsilon^{\beta}}} \left(\min\left\{ 1,  \frac{2 \eps}{\sqrt{c_2}\,\sigma(r_n)} \right\}\right)^d dt_n \le C n \eps^d f(\eps).
\end{align*}

\medskip
\noindent\textbf{Case (2)} $\gamma = 1/d$:
In this case, $f(r) = \log\log(1/r)$. We can use the properties of logarithm
to find that
\begin{align*}
A(\eps) = C\,\varepsilon^d\, \log\left(\frac{\log(1/\sigma^*(\eps))}{\log(1/\varepsilon^{\beta})}\right)\sim C\,\varepsilon^d\,\log\left(\frac{1}{H\beta}\right) \quad \text{as $\eps \to 0$.}
\end{align*}
Thus $A(\eps) \sim C \eps^d$. On the other hand $(\sigma^*(\eps))^N=C \eps^d\left(\log(1/\varepsilon)\right)^{-1}\ll \eps^d$. Then the entire bound simplifies to:
\begin{align*}
\int_{B_{\varepsilon^{\beta}}} \left(\min\left\{ 1,  \frac{2 \eps}{\sqrt{c_2}\,\sigma(r_n)} \right\}\right)^d dt_n \le C n \eps^d.
\end{align*}
Returning to \eqref{E:ETn2}, we apply the estimates derived above. For Case (1), we have $\E[T_\eps^n] \le (C n \eps^d f(\eps)) \E[T_\eps^{n-1}]$.
For Case (2), we have $\E[T_\eps^n] \le (C n \eps^d) \E[T_\eps^{n-1}]$.
The result \eqref{E:sojourn:UB-Psi} follows immediately by induction.
\end{proof}

Next, we turn to establishing a lower bound for $\E[T_\eps^n]$. 
%%we will make use of Lemmas \ref{L:PX:LB} and \ref{L:int_D:LB}. 
In Lemmas \ref{L:PX:LB} and \ref{L:int_D:LB} below, for any $t \in \R^N$ and any set $F \subset \R^N$, $d(t,F)$ denotes the distance defined by 
\begin{align*}
% \label{E:dist}
    d(t,F) = \inf\{|t-s| : s \in F\}.
\end{align*}

\begin{lemma}\label{L:PX:LB}
There exists a constant $C_0>0$ such that for all $\eps\in(0,1)$, all $n$ of the form $n = 2^p$ for some $p \in \N^+$, and all $t_1, \dots, t_n\in B_{\varepsilon^{\beta}}$,
\begin{align*}
% \label{E:PX:LB}
\P\left\{ |X(t_1)| \le \eps, \dots, |X(t_n)|\le \eps\right\} \ge C_0^n \eps^{nd} \frac{1}{\sigma^d(|t_1|)} \prod_{0 \le k < p} \prod_{2^k < i \le 2^{k+1}} \frac{1}{\sigma^d(d(t_i, F_k))}
\end{align*}
provided that
\begin{align}\label{E:cond:t_i}
\sigma(|t_1|) \ge 2^{-p}\eps \, \text{ and } \sigma(d(t_i, F_k)) \ge 2^{k-p}\eps \
    \text{ for  $0\le k < p$ and $2^k < i \le 2^{k+1}$,}
\end{align}
where $F_k = \{t_1, \dots, t_{2^k}\}$.
%and $d(t, F)$ is defined by \eqref{E:dist}.
\end{lemma}

\begin{proof}
For $0 \le k < p$ and $2^k < i \le 2^{k+1}$, let $a(i) \in \{1,\dots, 2^k\}$ be such that
\[
|t_i - t_{a(i)}| = d(t_i, F_k).
\]
Observe by the triangle inequality that if $|X(t_1)| \le 2^{-p} \eps$ and
\begin{align*}
|X(t_i) - X(t_{a(i)})| \le 2^{k-p}\eps \text{ for all $k$ and $i$ with $0\le k <p$ and $2^k< i \le 2^{k+1}$,}
\end{align*}
then $|X(t_i)| \le 2\eps$ for all $1 \le i \le 2^p$.
%Recall \v{S}id\'ak's theorem \cite{S68}: For any sequence $\{Z_j\}$ of jointly Gaussian centered random variables, we have
%\begin{align*}
%\P\{ |Z_j| \le \eps_j \text{ for all } j \} \ge \prod_j \P\{ |Z_j| \le \eps_j \}.
%\end{align*}
This, together with the Gaussian correlation inequality \cite{LM, Royen} and the fact that the coordinate processes $X_1,\dots, X_d$ are i.i.d.,~implies that
\begin{align*}
&\P\{ |X(t_1)| \le 2\eps, \dots, |X(t_n)| \le 2\eps  \}\\
& \ge \left( \P\left\{ |X_1(t_1)| \le \frac{2^{-p}\eps}{\sqrt{d}} \right\} \prod_{0 \le k < p}\prod_{2^k < i \le 2^{k+1}} \P\left\{ |X_1(t_i) - X_1(t_{a(i)})| \le \frac{2^{k-p}\eps}{\sqrt{d}}\right\}\right)^d.
\end{align*}
For $a \in (0, 1]$ and a standard normal random variable $Z$,
\begin{align*}
\P\{ |Z| \le a \} = \frac{2}{\sqrt{2\pi}} \int_0^a e^{-\frac{x^2}{2}} dx \ge c_0 a,
\end{align*}
where $c_0 = 2e^{-1/2}/\sqrt{2\pi}$. 
Condition \eqref{E:cond:t_i} ensures that
\[
\frac{2^{-p}\eps}{\sqrt{d}\,\sigma(|t_1|)} \le 1
\quad \text{and} \quad
\frac{2^{k-p}\eps}{\sqrt{d}\,\sigma(|t_i-t_{a(i)}|)} = \frac{2^{k-p}\eps}{\sqrt{d}\,\sigma(d(t_i,F_k))} \le 1.
\]
Hence, we have
\begin{align*}
\P\{ |X(t_1)| \le 2\eps, \dots, |X(t_n)| \le 2\eps  \}
\ge \left( \frac{c_0}{\sqrt{d}} \right)^{nd} \left( \frac{2^{-p}\eps}{\sigma(|t_1|)} \prod_{0 \le k<p} \prod_{2^k < i \le 2^{k+1}} \frac{2^{k-p}\eps}{\sigma(d(t_i, F_k))} \right)^d.
\end{align*}
To finish the proof, recall that $n=2^p$ and note that
\begin{align}\label{E:prod2k-p}
2^{-p} \prod_{0\le k < p} (2^{k-p})^{2^k}
\ge 2^{-n} 2^{\sum_{k=0}^{p-1} (k-p)2^k} = 2^{-n}2^{-\sum_{k=1}^p k 2^{p-k}} \ge c_1^n,
\end{align}
where $c_1 = 2^{-1-\sum_{k=1}^\infty k2^{-k}}$ is a positive constant. 
%This proves Lemma \ref{L:PX:LB}.
\end{proof}

Lemma \ref{L:int_D:LB} below shows that the set of points  $t_1, 
\dots, t_n\in B_{\varepsilon^\beta}$ that satisfy (\ref{E:cond:t_i}) is quite large.
This is essential for deriving a sharp lower bound for $\E[T_\eps^n]$ in 
Lemma \ref{L:sojourn:LB}. 

\begin{lemma}\label{L:int_D:LB}
% Let $\sigma$ given by \eqref{sigma:gamma} such that $\gamma\le 1/d$ and $N=Hd$.
There exist constants $\delta_2\in(0,1)$ and  $C_2\in(0,\infty)$ such that for
any $\eps \in (0,\delta_2)$ and any $n$ of the form $n=2^p$ for some 
$p\in\N^+$ with $n \le \log\log\log(1/\eps)$, there is a subset $D$ of 
$B_{\eps^{\beta}}^n$ with the following properties:
\begin{enumerate}
\item[\rm (i)] Every $(t_1, \dots, t_n) \in D$ satisfies 

\begin{align}\label{E:sigma-t1}
\sigma(|t_1|) \ge 2^{-p}\eps,
\end{align}
and 
\begin{align}\label{E:sigma-ti}
\sigma(d(t_i, F_k)) \ge 2^{k-p} 
\eps
\end{align}
 for all $k$ and $i$ 
with $0 \le k < p$ and $2^k < i \le 2^{k+1}$;
\item[\rm (ii)] Let $\mathcal{J}_n(D)$ denote the integral
    \[
    \mathcal{J}_n(D) := \int_D \frac{1}{\sigma^d(|t_1|)} \prod_{0 \le k < p} \prod_{2^k < i \le 2^{k+1}} \frac{1}{\sigma^d(d(t_i, F_k))} \, dt_1 \cdots dt_n.
    \]
Then, the following estimate holds:
    \begin{align}\label{E:int_D:LB}
    \mathcal{J}_n(D) \ge C_2^n 2^{-p} \prod_{0 \le k < p} \left[ (2^k)! \cdot \left( 2^{k-p} \, \Psi(\varepsilon) \right)^{2^k} \right],
    \end{align}
    where $\Psi$ is defined in \eqref{E:psi}.
\end{enumerate}
\end{lemma}

\begin{proof}
First, we construct a decreasing sequence $(\varepsilon_k)_{0 \le k \le p}$ in $(0,\varepsilon^{\beta})$ that satisfies the following key properties, for the two cases $\gamma<1/d$ and $\gamma=1/d$, respectively, for some $\delta_2 \in (0,1)$ and a constant $\mathbf{C} > 0$:
\begin{enumerate}
    \item[\textbf{(P1)}] $\varepsilon_{k+1} \le \varepsilon_k/8$ for all $k=0,\dots,p-1$ and $\varepsilon < \delta_2$.
    \item[\textbf{(P2)}] $\sigma(\varepsilon_1) \ge 2^{-p} \varepsilon$.
    \item[\textbf{(P3)}] $\sigma(\varepsilon_{k+1}) \ge 2^{k-p} \varepsilon$ for all $k=1,\dots,p-1$.
    \item[\textbf{(P4)}] $f(\varepsilon_{k+1}) - f(\varepsilon_k/4) \ge \mathbf{C} 2^{k-p} \Psi(\varepsilon)$.
\end{enumerate}

\emph{\textbf{Case 1}: $\gamma < 1/d$}. Define the sequence $(\eps_k)_{0\le k\le p}$ by setting $\eps_0=\eps^{\beta}$ and 
\begin{align*}
% \label{E:eps_k}
    \eps_{k}={\eps}^{{\mathbf{C}2^{k-p}}+\beta} \quad \text{for $k=1,\ldots,p$,}
\end{align*}
where $\mathbf{C}>0$ is a constant to be specified later. We now verify that 
\textbf{(P1)}--\textbf{(P4)} are satisfied. 

\textbf{(P1)} Let $k\in \{0,\ldots,p-1\}$.
Then 
\[
\varepsilon_{k+1}/\varepsilon_k=\varepsilon^{\mathbf{C}2^{k-p}}\le \varepsilon^{\mathbf{C}/n} \le \varepsilon^{\mathbf{C}/\log\log\log(1/\varepsilon)} =\operatorname{o}(1) \quad \text{as $\varepsilon \rightarrow 0$,}
\]
where we have used $n = 2^{p}$ and $n \le \log\log\log(1/\varepsilon)$ to obtain the two inequalities, respectively.
Then for some $\delta_2\in (0,1)$ small enough, we have 
\begin{align*}
    \eps_{k+1}/\eps_k\le 1/8 \quad\text{for all $\eps\in (0,\delta_2)$}.
\end{align*}    

\textbf{(P2)} Since $\gamma<1/d\le 1$ and $1<\beta<1/H$, then for any $0<\mathbf{C}<1/H-\beta$ we have 
\begin{align*}
\sigma(\varepsilon_1)=\varepsilon^{H(\mathbf{C}2^{1-p}+\beta)}\left(\left(\mathbf{C}2^{1-p}+\beta\right)\log(1/\varepsilon)\right)^{\gamma}\gtrsim \varepsilon^{H(\mathbf{C}+\beta)}2^{\gamma(1-p)}\ge 2^{-p}\varepsilon.
\end{align*}

\textbf{(P3)} In the same way, for $0<\mathbf{C}<\bigl(1/H-\beta\bigl)/2$, we have
\begin{align*}
\sigma(\varepsilon_{k+1})\gtrsim \varepsilon^{H(2\mathbf{C}+\beta)}2^{\gamma
(k-p)} \ge 2^{k-p}\varepsilon.
\end{align*}

\textbf{(P4)} Let $k\in \{0,\ldots,p-1\}$. Then using the definition \eqref{D:f_cases} of $f$ we have
\begin{align}\label{E:f-eps-4}
    &f(\varepsilon_{k+1})-f(\varepsilon_k/4)\nonumber\\
    &= \left(2\mathbf{C}2^{k-p}+\beta\right)^{1-\gamma d}\left(\log(1/\varepsilon)\right)^{1-\gamma d}
    - \left(\log(4)+(\mathbf{C}2^{k-p}+\beta)\log(1/\varepsilon)\right)^{1-\gamma d}\nonumber\nonumber\\
    &=\left(\log(1/\varepsilon)\right)^{1-\gamma d}\left[\left(2\mathbf{C}2^{k-p}+\beta\right)^{1-\gamma d}- \left(\frac{\log(4)}{\log(1/\varepsilon)}+(\mathbf{C}2^{k-p}+\beta)\right)^{1-\gamma d}\right]\nonumber\\
    &=f(\varepsilon)\left(\mathbf{C}2^{k-p}-\frac{\log(4)}{\log(1/\varepsilon)}\right)\frac{1-\gamma d}{(\beta+\Tilde{c})^{\gamma d}},
\end{align}
where $\Tilde{c}$ is such that $\mathbf{C}2^{k-p}+\frac{\log(4)}{\log(1/\varepsilon)}<\Tilde{c}<2\mathbf{C}2^{k-p}$, since $2^p\le \log\log\log(1/\varepsilon)$ and hence 
\begin{align}\label{E:mvt}
\mathbf{C}2^{k-p}+\frac{\log(4)}{\log(1/\varepsilon)}=\mathbf{C}2^{k-p}\left(1+ \operatorname{O}\left(\frac{2^p}{\log(1/\varepsilon)}\right)\right)=\mathbf{C}2^{k-p}\bigl(1+\operatorname{o}(1)\bigl).
\end{align}
In the same way, since $2^{p}\log(4)/\log(1/\varepsilon)=\operatorname{o}(1)$, combining \eqref{E:f-eps-4} and \eqref{E:mvt} we obtain
\begin{align*}
f(\varepsilon_{k+1})-f(\varepsilon_k/4)\ge f(\varepsilon)\mathbf{C} 2^{k-p}\left(1-\frac{2^{p}\log(4)}{\log(1/\varepsilon)}\right)\frac{1-\gamma d}{\beta+2\mathbf{C}}=\,\widetilde{\mathbf{C}} 2^{k-p}f(\varepsilon).
\end{align*}

\emph{\textbf{Case 2}: $\gamma = 1/d$}.
Define the sequence $(\eps_k)_{0\le k\le p}$ recursively by setting $\eps_0=\eps^{\beta}$ and 
\begin{align*}
% \label{E:eps_k}
    \eps_{k+1}=\Bigl(\frac{\eps_{k}}{4}\Bigl)^{e^{\mathbf{c}2^{k-p}}} \quad \text{for}\quad k=0,\ldots,p-1,
\end{align*}
where $\mathbf{c}>0$ is a constant which will be specified later. 

\textbf{(P1)} Note that $(\eps_k)$ is decreasing.
%and
%\begin{align}\label{E:eps-bd}
%    \eps_{k+1}\le \frac{\eps_k}{4}\le \frac{\eps^{\beta}}{4}\quad \text{for all}\quad k=0,\ldots,p-1
%\end{align}
Moreover, as $\eps \to 0$,
\begin{align*}
\frac{\eps_{k+1}}{\eps_k}= \frac{1}{4}\left(\frac{\eps_{k}}{4}\right)^{(e^{\mathbf{c}2^{k-p}}-1)}
\le \frac{1}{4}\left(\frac{\eps^{\beta}}{4}\right)^{(e^{\mathbf{c}/\log\log\log(1/\eps)}-1)}\le \frac{1}{4}\left(\frac{\eps^{\beta}}{4}\right)^{{\mathbf{c}/\log\log\log(1/\eps)}}=o(1).
\end{align*}
Then there is some $\delta_2\in (0,1)$ such that 
\begin{align*}
    \eps_{k+1}/\eps_k\le 1/8 \quad\text{for all $\eps\in (0,\delta_2)$}.
\end{align*}

\textbf{(P2)} Notice that $\eps_1=({\eps^{\beta}}/{4})^{e^{\mathbf{c}2^{-p}}}$. We require that $\mathbf{c}\le \log\bigl(\frac{1-\eta}{H\beta}\bigl)$ for some $\eta\in (0,1-H\beta)$, so that $H\beta e^{\mathbf{c}} \le 1-\eta$. Then for all $\eps$ small enough,
\[
\sigma(\eps_1) \ge \eps_1^H \ge  \left(\frac{\eps^{\beta}}{4}\right)^{He^{\mathbf{c}2^{-p}}}\ge \left(\frac{\eps}{4^{1/\beta}}\right)^{H\beta e^{\mathbf{c}}} \ge \left( \frac{\eps}{4^{1/\beta}}\right)^{1-\eta} \ge 2^{-p}\eps.
\]

\textbf{(P3)}  It is not hard to check that 
\begin{align*}
\eps_{k+1}=\left(\frac{\eps^{\beta}}{4}\right)^{e^{\mathbf{c}(2^{k+1}-1)2^{-p}}}\times\left(\frac{1}{4}\right)^{\sum_{\ell=1}^{k}e^{\mathbf{c}\left(\frac{2^{\ell}-1}{2^{\ell-1}}\right)2^{k-p}}}\quad\text{for all $k=1,\ldots, p-1$}.
\end{align*}
Therefore,
\begin{align*}
    \sigma(\eps_{k+1})&=\eps_{k+1}^H\left(\log(1/\eps_{k+1})\right)^{1/d}\\
    &\ge \left(\frac{\eps}{4^{1/\beta}}\right)^{H\beta e^{\mathbf{c}(2^{k+1}-1)2^{-p}}}\left(\frac{1}{4}\right)^{H\sum_{\ell=1}^{k}e^{\mathbf{c}\left(\frac{2^{\ell}-1}{2^{\ell-1}}\right)2^{k-p}}}\left(\beta e^{\mathbf{c}(2^{k+1}-1)2^{-p}}\log\left({4^{1/\beta}}/{\eps}\right)\right)^{1/d}\\
    &\ge C\,e^{(\mathbf{c} 2^{k-p}/d)}\,\eps\, \left(\frac{\eps}{4^{1/\beta}}\right)^{\left(H\beta e^{\mathbf{c}(1-2^{-p})}-1\right)}\left(\frac{1}{4}\right)^{H{k}e^{\mathbf{c}2^{k+1-p}}}(\log(1/\eps))^{1/d}\\
    &\ge C2^{k-p}\,\eps\,\left(\frac{\eps}{4^{1/\beta}}\right)^{\left(H \beta e^{\mathbf{c}(1-(\log\log\log(1/\eps))^{-1})}-1\right)}\left(\frac{1}{4}\right)^{He^{\mathbf{c}}\log_2(\log\log\log(1/\eps))}\\
    &\ge C 2^{k-p}\eps^{1-\eta}\left(\frac{1}{4}\right)^{He^{\mathbf{c}}\log_2(\log\log\log(1/\eps))}\\
    & \ge 2^{k-p}\eps
\end{align*}
uniformly for all $\eps\in (0,\delta_2)$, where we used the facts that $2^{k-p}\le(2^{k+1}-1)2^{-p}\le 1-2^{-p}$, $2^p\le \log\log\log(1/\eps)$, and the choice of $\mathbf{c}$ above, which ensures that
\begin{align*}
    \mathbf{c}\left(1-(\log\log\log(1/\eps))^{-1}\right)\le \log\left(\frac{1-\eta}{H\beta}\right).
\end{align*} 
% for some $\eta\in (0,1-H\beta)$ 

\textbf{(P4)} Let $k\in \{0,\ldots,p-1\}$. Then using the definition \eqref{D:f_cases} of $f$ we have
\begin{align*}
f(\varepsilon_{k+1})-f(\varepsilon_k/4)=\log\left( \frac{\log(1/\varepsilon_{k+1})}{\log(4/\varepsilon_k)} \right)= \log\left( \exp(\mathbf{c} 2^{k-p}) \right) = \mathbf{c} 2^{k-p}.
\end{align*}

Hence, the properties \textbf{(P1)}--\textbf{(P4)} are verified in both cases $\gamma<1/d$ and $\gamma=1/d$. Now, we proceed to construct the set $D$. For $t\in B_{\varepsilon^{\beta}}$ and $k=0,1,\dots,p-1$, let
\begin{align*}
H_k(t) = \{ s \in \R^N : \eps_{k+1} \le |s-t| \le \eps_k/4 \}
\end{align*}
be the spherical shells centered at $t$. For $k \ge 0$, let
\begin{align*}
    A_k = \left\{ a:\{2^k+1, 2^k+2, \dots, 2^{k+1}\} \to \{1,\dots, 2^k\} \ | \ a \text{ is bijective} \right\}.
\end{align*}
Note that the cardinality of $A_k$ is
\begin{align}\label{E:card:A_k}
    \# A_k = 2^k!.
\end{align}
Define
\begin{align}\begin{split}\label{E:D}
D = \Big\{ (t_1, \dots, t_n) & : t_1 \in H_0(0), \text{and for every 
$0 \le k < p$,} \\
& \quad (t_{2^k+1}, \dots, t_{2^{k+1}}) \in \bigcup_{a_k \in A_k} 
\left(H_k(t_{a_k(2^k+1)}) \times \cdots \times 
H_k(t_{a_k(2^{k+1})})\right) \Big\}.
\end{split}\end{align}
By \textbf{(P1)}, we see that if $(t_1,\dots,t_n)\in D$, then for 
every $0\le k < p$, there exists $a_k \in A_k$ such that for all $2^k 
< i \le 2^{k+1}$,
\begin{align*}
|t_i| &\le |t_i-t_{a_k(i)}| + |t_{a_k(i)}|\\
&\le |t_i-t_{a_k(i)}| + |t_{a_k(i)}-t_{a_{k-1}(a_k(i))}|+|t_{a_{k-1}
(a_k(i))}|\\
& \le |t_i-t_{a_k(i)}| + \sum_{j=0}^{k-1} |t_{a_{j+1}(\cdots(a_k(i)))}
-t_{a_j(\cdots (a_k(i)))}|\\
& \le \sum_{j=0}^k \frac{\eps_j}{4} 
\le \frac{1}{4}\sum_{j=0}^k \frac{\eps_0}{4^j} \le \frac{\eps^{\beta}}{3} \le \eps^{\beta},
\end{align*}
since  $\eps_0 = \eps^{\beta}$. This verifies that $D\subset B_{\eps^{\beta}}^n$.

Next, we verify that $D$ satisfies the desired properties (i) and 
(ii) in the lemma for both cases $\gamma <1/d$ and $\gamma = 1/d$.  Let $(t_1, \dots, t_n) \in D$. Then by {\bf (P2)}, $\sigma(|t_1|) \ge \sigma(\eps_1) \ge  2^{-p}\eps$. This shows \eqref{E:sigma-t1}.
In order to show \eqref{E:sigma-ti}, we claim that 
\begin{align}\label{d(t,F):bd}
{\eps_{k+1} \le d(t_i, F_k) \le \eps_k/4} \quad \text{for } 0 \le k < p \text{ and } 2^k < i \le 2^{k+1}
\end{align}
and
\begin{align}\label{H_k:disjoint}
    \text{for each $0 \le k < p$, } 
    H_k(t_{1}), \dots,  H_k(t_{2^k}) \text{ are pairwise disjoint.}
\end{align}
In fact, the right-hand inequality of \eqref{d(t,F):bd} follows 
immediately from $t_i\in H_k(t_{a(i)})$ according to the definition of 
$D$ in \eqref{E:D}.
The left-hand inequality of \eqref{d(t,F):bd} can be proved by 
induction. 
Indeed, when $k = 0$ and thus $i = 2$, we see that for any $t_2\in 
H_0(t_1)$,
\[
d(t_2, F_0) = |t_2 - t_1| \in [\eps_1, \eps_0/4].
\]
For the induction hypothesis, we assume that for certain $k$ where $0\le k < p$,
\begin{align}\begin{split}\label{E:int_D:ind:hyp}
&\eps_{k+1} \le d(t_i, F_k) \le \eps_k/4 \quad \text{for all $i$ with $2^k < i \le 2^{k+1}$ and}\\
&\text{$|t_l - t_j| \ge \eps_k$ for all $l, j \in \{1, \dots, 2^k\}$ with $l \ne j$.}
\end{split}\end{align}
We first show that the second part of \eqref{E:int_D:ind:hyp} holds when $k$ is replaced by $k+1$, so let us consider $l' \ne j' \in \{1, \dots, 2^{k+1}\}$.
This is certainly true if both $l', j' \in \{1,\dots, 2^k\}$, so we now consider $l' \in \{1,\dots, 2^{k+1}\}$ and $j' \in \{2^k+1,\dots, 2^{k+1}\}$.
In particular, if $l' \in \{1, \dots, 2^k\}$, then by induction hypothesis \eqref{E:int_D:ind:hyp},
\[
    |t_{l'}-t_{j'}| \ge d(t_{j'}, F_k) \ge \varepsilon_{k+1};
\]
and if $l' \in \{2^k+1,\dots, 2^{k+1}\}$, then by the triangle inequality, 
$a(l')\ne a(j')$, induction hypothesis \eqref{E:int_D:ind:hyp}, and \textbf{(P1)}, we have
\begin{align*}
% |t_{l'} - t_{j'}| &\ge \min\{|t_m-t_{j'}| : 1 \le m \le 2^k \} - d(t_{l'}, F_k)\\
% &\ge \eps_k - \eps_k/4\\
% & \ge \eps_{k+1}.
|t_{l'} - t_{j'}|
& \ge |t_{a(l')} - t_{a(j')}| - |t_{l'}-t_{a(l')}| - |t_{j'} - t_{a(j')}|\\
& \ge \varepsilon_k - \varepsilon_k/4 - \varepsilon_k/4 = \varepsilon_k/2\\
&\ge \varepsilon_{k+1}.
\end{align*}
Hence, in any case, for $2^{k+1} < i \le 2^{k+2}$,
\begin{align*}
d(t_i, F_{k+1}) &\ge \min\{ |t_m - t_{a(i)}| : 1 \le m \le 2^{k+1} \} - |t_i-t_{a(i)}|\\
& \ge \eps_{k+1} - \eps_{k+1}/4\\
& \ge \eps_{k+2}.
%& \ge \sigma^*(\eps_{k+2}),
\end{align*}
%where we have used \eqref{E:r/sigma^*} to obtain the last inequality.
By induction, \eqref{E:int_D:ind:hyp} holds for all $0 \le k < p$ and the claim \eqref{d(t,F):bd} follows.
Also, the second part of \eqref{E:int_D:ind:hyp} together with the triangle inequality implies property \eqref{H_k:disjoint}.

Now, by \eqref{d(t,F):bd} and \textbf{(P3)} we have that for $1 \le k < p$ and
$2^k < i \le 2^{k+1}$,
\begin{align*}
{\sigma(d(t_i,F_k)) \ge \sigma(\eps_{k+1})\ge 2^{k-p}\eps,}
\end{align*}
which is \eqref{E:sigma-ti}.
The proof of (i) is now complete.

It remains to verify the estimate in (ii) of Lemma \ref{L:int_D:LB}.
The property \eqref{H_k:disjoint} above implies that for every $0\le k < p$,
\begin{align*}
    \bigcup_{a_k \in A_k}\left(H_k(t_{a_k(2^k+1)})\times \dots \times H_k(t_{a_k(2^{k+1})})\right)
\end{align*}
is a disjoint union. Indeed, if $a_k \neq a'_k\in A_k$, there exists $ 2^k+1 \le i\le 2^{k+1}$ such that $a_k(i) \neq a'_k(i)$. By \eqref{H_k:disjoint}, the sets $H_k(t_{a_k(i)})$ and $H_k(t_{a'_k(i)})$ are disjoint, meaning the Cartesian products are disjoint at their $i$-th coordinate.
Then, recall the definition of $D$ in \eqref{E:D} and use the preceding disjointness property to write
\begin{align*}
&\mathcal{J}_n(D)\\
&=\quad \int_D \frac{1}{\sigma^d(|t_1|)} \prod_{0 \le k < p} \prod_{2^k < i \le 2^{k+1}} \frac{1}{\sigma^d(d(t_i, F_k))} dt_1 \cdots dt_n\\
&= \int_{H_0(0)} \frac{dt_1}{\sigma^d(|t_1|)} 
\prod_{0 \le k < p} \int_{\bigcup_{a_k \in A_k}\left(H_k(t_{a_k(2^k+1)})\times \dots \times H_k(t_{a_k(2^{k+1})})\right)} \prod_{2^k<i\le 2^{k+1}}\frac{1}{\sigma^d(t_i,F_k)}dt_{2^k+1}\cdots dt_{2^{k+1}}\\
&= \int_{H_0(0)} \frac{dt_1}{\sigma^d(|t_1|)} \prod_{0\le k < p} \sum_{a_k \in A_k} \prod_{2^k < i \le 2^{k+1}} \int_{H_k(t_{a_k(i)})} \frac{dt_i}{\sigma^d(d(t_i,F_k))} .
\end{align*}
We integrate in the order $dt_n, dt_{n-1},\dots, dt_1$.
For fixed $t_1,\dots,t_{2^k}$ ($0\le k < p$), we use the obvious inequality 
$d(t_i,F_k) \le |t_i-t_{a_k(i)}|$ for $2^k < i \le 2^{k+1}$, then use the 
polar coordinate, the definitions of $L$ and $f$ in \eqref{E:sigma-def} and 
\eqref{D:f_cases}, and the property \textbf{(P4)} to deduce that for all 
$a_k \in A_k$ and for all $i$ with $2^k<i\le 2^{k+1}$,
\begin{align*}
\int_{H_k(t_{a_k(i)})} \frac{dt_i}{\sigma^d(d(t_i, F_k))} 
&\ge \int_{H_k(t_{a_k(i)})} \frac{ dt_i}{\sigma^d(|t_i - t_{a_k(i)}|)} 
= C \int_{\eps_{k+1}}^{\eps_k/4} \frac{d\rho}{\rho L^d(\rho)} \\
&=C \left(f(\eps_{k+1}) - f(\eps_k/4)\right)
\ge \mathbf{C} 2^{k-p},
\end{align*}
where the constant $C$ does not depend on $k$, $i$, $t_1,\dots,t_{2^k}$, or $a_k \in A_k$.
Similarly, we have
\begin{align*}
    \int_{H_0(0)} \frac{dt_1}{\sigma^d(|t_1|)} =C\bigl(f(\eps_1)-f(\eps_0)\bigl) \ge \mathbf{C}\, 2^{-p}.
\end{align*}
Therefore, the above estimates and \eqref{E:card:A_k} lead to \eqref{E:int_D:LB} for some uniform constant $C_2$.
This completes the proof of Lemma \ref{L:int_D:LB}.
\end{proof}

\begin{lemma}\label{L:sojourn:LB}
% Let $\sigma$ be given as in \eqref{E:sigma-def}.
There exist constants $\delta_2\in(0,1)$ and $C_3\in(0,\infty)$ such that for all $\eps \in (0, \delta_2)$ and all $n$ of the form $n = 2^p$ for some $p\in \N^+$ with $n \le \log\log\log(1/\eps)$,
\begin{align}\label{E:sojourn:LB}
\E[T_\eps^n] \ge (C_3 n\eps^d\Psi(\eps))^n.
\end{align}
\end{lemma}

\begin{proof}
Choose $\delta_2, C_2, C_3$ and the subset $D\subset B_{\eps^{\beta}}^n$ according to Lemma \ref{L:int_D:LB}.
Properties (i) and (ii) in Lemma \ref{L:int_D:LB} combined with Lemma \ref{L:PX:LB} lead to the following:
\begin{align*}
\E[T_\eps^n] &= \int_{B_{\varepsilon^{\beta}}^n} \P\{ |X(t_1)| \le \eps , \dots, |X(t_n)| \le \eps \} dt_1 \cdots dt_n\\
& \ge C_0^n \eps^{nd} \int_D \frac{1}{\sigma^d(|t_1|)} \prod_{0 \le k < p} \prod_{2^k < i \le 2^{k+1}} \frac{1}{\sigma^d(d(t_i, F_k))} dt_1 \cdots dt_n\\
& \ge C_0^n C_2^n \eps^{nd} 2^{-p}  \prod_{0 \le k < p}\left[ 2^k! \left( 2^{k-p} \Psi(\varepsilon) \right)^{2^k} \right].\\
& = C_0^n C_2^n \bigl(\eps^{d}\Psi(\eps)\bigl)^n  2^{-p} \prod_{0 \le k < p} \left[2^k!\, 2^{(k-p)2^k}\right].
\end{align*}
By Stirling's formula, we have
\begin{align*}
    \prod_{0 \le k < p} 2^k! \ge \prod_{0 \le k < p} c_0^k 2^{k2^k}
\end{align*}
for some constant $0<c_0<1$.
By differentiating the identity $\sum_{k=0}^{p-1} x^k = (x^p-1)/(x-1)$, multiplying by $x$, and then putting $x=2$, we can deduce that $\sum_{k=0}^{p-1}k2^k = p2^p-2^{p+1}+2$.
It follows that
\begin{align}\label{E:prod:2^k!}
    \prod_{0 \le k < p} 2^k! \ge c_0^{p(p-1)/2} 2^{p2^p}2^{-2^{p+1}} \ge c_0^{2^p}2^{p2^p}4^{-2^p} = (c_0/4)^n n^n,
\end{align}
where we have used $p(p-1)/2 \le 2^p$ in the second inequality and $2^p=n$ in the last equality. 
Finally, we can apply \eqref{E:prod:2^k!} and \eqref{E:prod2k-p} to the lower bound for $\E[T_\varepsilon^n]$ above to obtain \eqref{E:sojourn:LB} with constant $C_3 = C_0C_2c_1c_0/4$.
\end{proof}

Recall the Paley--Zygmund inequality: for any nonnegative random variable $Y$ and any constant $\theta\in[0,1]$,
\begin{align}\label{E:PZ}
\P\{ Y \ge \theta\E[Y]\} \ge (1-\theta)^2\frac{(\E[Y])^2}{\E[Y^2]}.
\end{align}

\begin{proposition}\label{P:sojourn:LD}
% Let $\sigma$ be given as in \eqref{E:sigma-def}.
There exist constants $\eps_0\in(0,1)$ and $K_1, K_2\in(0,\infty)$ such that
\begin{align}\label{E:sojourn:LD}
\P\{ T_\eps \ge u\eps^d \Psi(\eps) \} \ge \tfrac14 e^{-K_1 u}
\end{align}
for all $\eps \in (0, \eps_0)$ and $1 \le u \le K_2 \log\log\log(1/\eps)$.
\end{proposition}

\begin{proof}
Take $\eps_0 = \min\{\delta_1,\delta_2\}$, where $\delta_1$ and $\delta_2$ are the constants given by Lemmas \ref{L:sojourn:UB} and \ref{L:sojourn:LB}.
Take $K_2 = (C_2 \wedge 2)/4$, where $C_2$ is the constant given by Lemma \ref{L:int_D:LB}.
Let $\eps\in(0,\eps_0)$ and $1\le u \le K_2 \log\log\log(1/\eps)$.
Since $2u/(C_2\wedge 2) \ge 1$, we can find $p\in\N^+$ such that $2^{p-1} \le 2u/(C_2 \wedge 2) \le 2^p$.
Set $n=2^p$.
Note that $n \le 4u/(C_2 \wedge 2) \le \log\log\log(1/\eps)$ and $u \le 
C_2 n/2$.
Then, by Lemma \ref{L:sojourn:LB} and the Paley--Zygmund inequality \eqref{E:PZ} with $\theta=1/2$,
\begin{align*}
&\P\{ T_\eps \ge u \eps^d \Psi(\eps) \} \\
&\ge \P\left\{ T_\eps^n \ge \tfrac12\left(C_2 n\eps^d \Psi(\eps)\right)^n\right\}
\ge \P\left\{ T_\eps^n \ge \tfrac12\E[T_\eps^n]\right\}
\ge \frac{(\E[T_\eps^n])^2}{4\,\E[T_\eps^{2n}]}.
\end{align*}
Applying the moment estimates of the sojourn time in Lemmas \ref{L:sojourn:UB} and \ref{L:sojourn:LB}, we get that
\begin{align*}
\P\{ T_\eps \ge u \eps^d \Psi(\eps) \} 
& \ge \frac{\left(C_3 n\eps^d \Psi(\eps)\right)^{2n}}{4\left(2C_1 n \eps^d \Psi(\eps)\right)^{2n}} = \frac14\left( \frac{C_3}{2C_1}\right)^n \ge \tfrac14 e^{-C_4 n}
    \end{align*}
for some constant $C_4>\log(2C_1/C_3)$.
Since $n \le 4u/(C_2\wedge 2)$, we obtain \eqref{E:sojourn:LD} with $K_1:=4C_4/(C_2\wedge 2)$.
\end{proof}

\section{Proof of Theorem \ref{T:Haus:meas}}
\label{S:pf:T2}

Throughout this section, we let Assumptions \ref{a:X} and \ref{a:SLND} hold with $\sigma$ given by \eqref{sigma:gamma} where $N=Hd$ and $\gamma \le 1/d$. 
Recall that the covariance function satisfies \eqref{E:cov}.
This implies that $X$ has the following spectral representation:
\begin{align}\label{E:spectral2}
X_j(t) = \int_{\R^N} (e^{it\cdot \xi} -1) W_j(d\xi), \quad j = 1, \dots, d,
\end{align}
where $W_1, \dots, W_d$ are i.i.d.\ centered complex-valued Gaussian random measures whose \emph{control measure} is the spectral measure $m$ in \eqref{E:cov}, such that
\[
	\E[W_1(A)\overline{W_1(B)}] = m(A\cap B) \quad \text{and} 
    \quad 	W_1(-A) = \overline{W_1(A)}
\]
for all Borel sets $A, B \subset \R^N$ with finite $m$-measure.
% The measure $m$ is called the \emph{spectral measure} of $X$.

In order to create independence, define, for $0\le a \le b$, the truncated Gaussian random field $X(a, b) = \{ X(a,b,t)=(X_1(a,b,t), \ldots, X_d(a,b,t)), t \in \R^N\}$ by
\begin{align}\label{E:X(a,b,t)}
    X_j(a,b,t) = \int_{a\le|\xi|<b} (e^{it\cdot\xi}-1) W_j(d\xi), \quad t \in \R^N, j=1,\dots, d.
\end{align}
Recall $\sigma^*$ defined in \eqref{D:sigma*}.
The following lemma quantifies the approximation error between the  Gaussian random fields $X(a, b)$ and $X$, which is an extension of Lemma 3.2 in \cite{T95}.
\begin{lemma}\cite[Lemma 3.3 and Corollary 3.1]{X96}\label{L:error:bd}
    There exist constants $K_0>0$, $B>0$ such that for any $B<a<b$ and $0<r<B^{-1}$, the following holds:
    let $A = r^2 a^2 \sigma^2(a^{-1})+\sigma^2(b^{-1})$ such that $\sigma^*(\sqrt{A}) \le r/2$, then for any
    \begin{align*}
        u \ge K_0 \left(A \log \frac{K_0 r}{\sigma^*(\sqrt{A})} \right)^{1/2},
    \end{align*}
    we have
    \begin{align*}
        \P\left\{ \sup_{|t|\le r} |X(t)-X(a,b,t)| \ge u \right\} \le \exp\left(-\frac{u^2}{K_0 A} \right).
    \end{align*}
\end{lemma}

The following lemma is essential for us to construct an economic random
covering for $X(I)$ to prove Theorem \ref{T:Haus:meas}.
\begin{lemma}\label{L:P:Talag}
Let $R_p = 2^{-2^{2^p}}$.
Then, there exist $\beta_0 \in (1,1/H)$, $c_0>0$ and $n_0, p_0 \in \N^+$ such that for all $\beta \in [\beta_0, 1/H)$, $p \ge p_0$ and $n_0 \le n \le p$, we have
\begin{align}\begin{split}\label{E:P:Talag}
\P\Big\{ \exists \, r \in [R_{2p}, R_p] \text{ such that } \lambda_N\{ t \in \R^N : |t| \le r^\beta, |X(t)| \le 3r \} \ge c_0 n r^d \Psi(r) \Big\}\\
\ge 1 - 2^{- (1+2d) 2^{2p-n+n_0}},
\end{split}\end{align}
where $\Psi$ is defined in \eqref{E:psi}.
\end{lemma}

\begin{proof}
Let  $1<\beta<1/H$. First, Proposition \ref{P:sojourn:LD} ensures that there exist $0<r_0<1$ and $K_1, K_2>0$ such that for all $r\in (0,r_0)$ and $1 \le u\leq K_2 \log\log\log(1/r)$,
\begin{align}\label{Cl:rescale}
\mathbb{P}\Bigl\{\lambda_N\bigl\{ t \in \R^N : |t| \le r^\beta, |X(t)| \le r \bigl\}\geq u {r}^d\Psi(r)  \Bigl\}\geq\tfrac{1}{4}e^{-K_1 u}.
\end{align}
Let $\zeta \in (1,2)$ be a number close to 1 whose value will be 
determined later. We choose $1<\beta'<\beta<1/H$ (depending on $\zeta$)
with $\beta'$ close to 1, $\beta$ close to $1/H$ such that
\begin{align}\label{beta/beta'}
    \frac{\beta}{\beta'}>\frac{1}{\zeta}\left(\frac{1}{H}-1\right)+1.
\end{align}
This is possible since $\zeta>1$ implies that the right-hand side is $<1/H$ while the left-hand side increases to $1/H$ as $\beta \uparrow 1/H$ and $\beta' \downarrow 1$.
Define 
\begin{align}
    r_\ell = 2^{-\zeta^\ell},\qquad 
    a_\ell = r_\ell^{-(\beta-\beta')/(1-H)},\qquad
    b_\ell = r_\ell^{-\beta'/H},
\end{align}
and $A_\ell = r_\ell^{2\beta} a_\ell^2 \sigma^2(a_\ell^{-1}) + 
\sigma^2(b_\ell^{-1})$. Notice that
\begin{align}\label{R<r<R}
    R_{2p} \le r_\ell \le R_p \ \text{ is equivalent to }\ 2^{2^p} \le \zeta^\ell \le 2^{2^{2p}}.
\end{align}
Since $\beta/\beta' < 1/H$, we have $a_\ell < b_\ell$ for all $\ell \in \N^+$. Moreover, since $r_{\ell+1} = r_\ell^\zeta$, and \eqref{beta/beta'} implies $\zeta(\beta-\beta')/(1-H)>\beta'/H$, it follows that 
\begin{align}\label{b<a}
    b_\ell \le a_{\ell+1} \quad \text{for all $\ell \in \N^+$.}
\end{align}
Recall the form of $\sigma$ and $L$ in \eqref{E:sigma-def}.
If we choose $\beta''$ such that $1<\beta''<\beta'$, then
\begin{align}\begin{split}\label{A<r}
    A_\ell &\le r_\ell^{2\beta} r_\ell^{-2(\beta-\beta')}L^2(r_\ell^{(\beta-\beta')/(1-H)}) + r_\ell^{2\beta'} L^2(r_\ell^{\beta'/H})\\
    &\le 2r_\ell^{2\beta''} \quad \text{for $\ell$ large.}
\end{split}\end{align}
It is possible to choose a constant $c_0 >0$ such that if $1 \le n \le p$ and $2^{2^p} \le \zeta^\ell \le 2^{2^{2p}}$, then 
\begin{align}\label{c_0}
    c_0n \le K_2 \log\log\log(1/r_\ell) \quad \text{and} \quad e^{K_1 c_0} \le 2, 
\end{align}
where $K_1, K_2$ are the constants that ensure \eqref{Cl:rescale} holds.
Recall the truncated process $X(a_\ell,b_\ell,t)$ introduced in \eqref{E:X(a,b,t)}, and define the events $E_\ell, F_\ell,$ and $ G_\ell$ by
\begin{align*}
    E_\ell&=\left\{ \sup_{|t|\le r_\ell^\beta} |X(t)-X(a_\ell,b_\ell,t)| \ge r_\ell \right\},\\
    F_\ell &= \left\{ \lambda_N\{ |t| \le r_\ell^\beta : |X(t)| \le r_\ell \} \ge c_0 n r_\ell^d \Psi(r_\ell)  \right\},\\
    G_\ell &= \left\{ \lambda_N\{ |t| \le r_\ell^\beta : |X(a_\ell,b_\ell,t)| \le 2r_\ell \} \ge c_0 n r_\ell^d \Psi(r_\ell) \right\}.
\end{align*}
To apply Lemma \ref{L:error:bd}, we need to verify the conditions for the choices $r:=r_{\ell}^{\beta}$, $u:=r_\ell$, $a:=a_\ell$, $b:=b_\ell$, and $A:=A_\ell$. In fact, applying $\sigma^*$ to both sides of \eqref{A<r}, we obtain $\sigma^*(\sqrt{A_\ell})\le r_{\ell}^{\beta''/H}\le r_{\ell}^{\beta}/2$, since $H\beta<1<\beta''$. 
On the other hand, from the definition of $A_{\ell}$ we trivially have $A_\ell \ge \sigma^2(b_\ell^{-1})$, which implies $\sigma^*(\sqrt{A_\ell}) \ge b_\ell^{-1} = r_\ell^{\beta'/H}$. Combining these bounds, for $\ell$ large, we have
\begin{align*}
K_0\Bigl(A_{\ell}\log \frac{K_0 r_\ell^\beta}{\sigma^*(\sqrt{A_\ell})}\Bigl)^{1/2} \le r_{\ell}^{\beta''}\bigl(\log \bigl( K_0 r_\ell^{\beta - \beta'/H} \bigl)\bigl)^{1/2} \le C\, r_{\ell}^{\beta''} (\log(1/r_\ell))^{1/2}\le r_{\ell}.
\end{align*}
Then, thanks to Lemma \ref{L:error:bd}, for $\ell$ 
large, we have
\begin{align*}
% \label{E:error:bd}
    \P(E_\ell)
    \le \exp\left( -\frac{1}{2K_0 r_\ell^{2\beta''-2}} \right).
\end{align*}
Owing to \eqref{R<r<R}, we can choose $p_1$ large enough so that if $p
\ge p_1$, then
\begin{align}\label{P(E)}
    \P(E_\ell) \le \exp\left(-\tfrac{1}{2K_0} 2^{\zeta^\ell(2\beta''-2)}\right) 
    \le \exp\left(-\tfrac{1}{2K_0} 2^{2^{p}(2\beta''-2)}\right) \le 2^{-n-3}
\end{align}
uniformly for all $n \le p$ and $\ell$ such that $2^{2^p} \le \zeta^\ell \le 2^{2^{2p}}$.
Because of \eqref{c_0}, we may apply \eqref{Cl:rescale} with $u= c_0 n$ to see that
\begin{align*}
    \P(F_\ell) \ge 2^{-2} e^{-K_1 c_0 n} \ge 2^{-n-2}.
\end{align*}
Since $F_\ell \cap E_\ell^c \subset G_\ell$, it follows that
\begin{align*}
    \P(F_\ell) \le \P(E_\ell) + \P(F_\ell \cap E_\ell^c) \le \P(E_\ell) + \P(G_\ell).
\end{align*}
Hence, for $1 \le n \le p$ and $2^{2^p} \le \zeta^\ell \le 2^{2^{2p}}$,
\begin{align}\label{P(G)}
    \P(G_\ell) 
    \ge \P(F_\ell) - \P(E_\ell)
    \ge 2^{-n-2} - 2^{-n-3} = 2^{-n-3}.
\end{align}
Let $A$ denote the event in \eqref{E:P:Talag}, i.e.,
\[
    A = \Big\{ \exists \, r \in [R_{2p}, R_p] \text{ such that } \lambda_N\{ |t| \le r^\beta : |X(t)| \le 3r \} \ge c_0 n r^d \Psi(r)  \Big\}.
\]
Then, by \eqref{R<r<R},
\begin{align}\label{P(A)}
    \P(A)
    &\ge \P\left\{ \exists \ell, 2^{2^p} \le \zeta^\ell \le 2^{2^{2p}} \text{ such that } \lambda_N\{ |t| \le r_\ell^\beta : |X(t)| \le 3r_\ell \} \ge c_0 n r_\ell^d \Psi(r_\ell) \right\}\nonumber\\
    & \ge \P\left( \bigcup_{\ell} (G_\ell \cap E_\ell^c) \right)
    \ge \P\left( \bigg(\bigcup_{\ell} G_\ell \bigg) \cap \bigg( \bigcap_{\ell} E_\ell^c \bigg) \right)\\\nonumber
    & \ge \P\left( \bigcup_{\ell} G_\ell \right) - \P\left( \bigcup_{\ell} E_\ell \right),
\end{align}
where $\ell$ runs through all integers such that $C_\zeta 2^p \le \ell \le C_\zeta 2^{2p}$ with $C_\zeta = (\log 2)/(\log \zeta)$.
Thanks to \eqref{b<a}, the processes $X(a_\ell, b_\ell, \cdot)$, $\ell \in \N^+$, are independent, which ensures that the events $\{G_\ell : C_\zeta 2^p \le \ell \le C_\zeta 2^{2p}\}$ are independent.
Hence, by \eqref{P(G)} and the elementary inequality $1-x\le \exp(-x)$, we deduce that for all $p \ge p_1$ and $n \le p$,
\begin{align*}
    \P\left( \bigcup_{C_\zeta 2^p \le \ell \le C_\zeta 2^{2p}} G_\ell \right)
    &= 1 - \prod_{C_\zeta 2^p \le \ell \le C_\zeta 2^{2p}} \left( 1-\P(G_\ell)\right)\\
    &\ge 1 - \left(1-2^{-n-3}\right)^{C_\zeta (2^{2p}-2^p)}\\
    & \ge 1 - \exp\left(-\tfrac{\log 2}{\log \zeta}(2^{2p}-2^p) 2^{-n-3}\right).
\end{align*}            
By choosing $1<\zeta<2$ close enough to 1, we can ensure there exists $n_0 \in \N^+$ such that for all 
$p \ge p_1$ and $n_0 \le n \le p$,
\begin{align}\label{P(UG)}
    \P\left( \bigcup_{C_\zeta 2^p \le \ell \le C_\zeta 2^{2p}} G_\ell \right)
    \ge 1 - \tfrac12 2^{-(1+2d)2^{2p-n+n_0}}.
\end{align}
Now, the uniform estimate \eqref{P(E)} ensures, for a sufficiently large $p_0 \ge p_1$, that 
\begin{align}\label{E:P(UE)}
\P\left( \bigcup_{\ell} E_\ell \right) \le C_\zeta \bigl(2^{2p}-2^{p}\bigl) \exp\left(-\frac{1}{2K_0}2^{2^{p}(2\beta''-2)}\right)\le \tfrac12 2^{-(1+2d)2^{2p-n+n_0}},
\end{align}
for all 
$p \ge p_0$ and $n_0 \le n \le p$. Putting \eqref{P(UG)} and \eqref{E:P(UE)} into \eqref{P(A)} 
yields
\begin{align*}
    \P(A) \ge 1 - \tfrac12 2^{-(1+2d)2^{2p-n+n_0}} - \tfrac12 2^{-(1+2d)2^{2p-n+n_0}} = 1-2^{-(1+2d)2^{2p-n+n_0}}.
\end{align*}
This completes the proof of Lemma \ref{L:P:Talag}.  
\end{proof}

Recall Vitali's covering lemma:

\begin{lemma}\cite[p.24, Theorem 2.1]{Mattila}\label{L:Vitali}
Given a family of closed balls $\mathscr{F}$ in $\R^d$ with bounded radius, there is
a disjoint subfamily $\mathscr{F}'$ of $\mathscr{F}$ such that the family $\mathscr{F}'' = \{ 5B : B \in \mathscr{F}' \}$ covers $\mathscr{F}$, where $cB$ denotes the ball with the same center as $B$ but whose radius is $c$ times the radius of $B$.
\end{lemma}

We are ready to prove Theorem \ref{T:Haus:meas}.

\begin{proof}[Proof of Theorem \ref{T:Haus:meas}]
We extend Talagrand's sojourn-time based covering argument in \cite{T98} to prove the theorem.
Fix a compact interval $I$ in $\R^N$.
For any $r \ge 0$, define 
\[
    I_r = \{ t \in \R^N : \inf_{s \in I}|t-s| \le r \}.
\]
Let $R_p$, $\beta$, $c_0>0$ and $n_0,p_0 \in \N^+$ be given by Lemma \ref{L:P:Talag}.
For any $p \in \N^+$, define the random sets $U_p$, $V_p$ by
\begin{align*}
    &U_p = \Big\{ t \in I_2 : \exists\, r \in [R_{2p},R_p], \lambda_N\{ s \in B(t,r^{\beta}) : |X(t)-X(s)| \le 4r \} \ge c_0 p\, r^d \Psi(r) \Big\},\\
    &V_p = \Big\{ t \in I_1 : \exists\, r \in [R_{4p},R_{2p}], \lambda_N\{ s \in B(t,r^{\beta}) : |X(t)-X(s)| \le 4r \} \ge c_0 n_0 r^d \Psi(r) \Big\},
\end{align*}
where $\Psi$ is given by \eqref{E:psi}, and define the events $\Omega_{p,1}$ and $\Omega_{p,2}$ by
\begin{align}\begin{split}\label{Omegap1:Omegap2}
    &\Omega_{p,1} = \{ \lambda_N(U_p) \ge (1-2^{-2^p}) \lambda_N(I_2) \},\\
    &\Omega_{p,2} = \{ \lambda_N(V_p) \ge (1-2^{-(1+d)2^{4p}}) \lambda_N(I_1) \}.
\end{split}\end{align}
By Markov's inequality, Fubini's theorem, Lemma \ref{L:P:Talag}, and the stationarity of increments of $X$, for all $p\ge p_0$
\begin{align*}
    \P\{ \Omega_{p,1}^c\} &= \P\{ \lambda_N(I_2\setminus U_p) > 
    2^{-2^p} \lambda_N(I_2) \}\\
    &\le \frac{1}{2^{-2^p}\lambda_N(I_2)} \E[\lambda_N(I_2 \setminus U_p)]\\
    &= \frac{1}{2^{-2^p}\lambda_N(I_2)} \int_{I_2} \P\{ t \not\in U_p \} dt\\
    & \le \frac{1}{2^{-2^p}\lambda_N(I_2)} 2^{-(1+2d)2^{2p-p+n_0}}\lambda_N(I_2).
\end{align*}
%%%%here we use Lemma \ref{L:P:Talag} with $n=p$ (upper endpoint)%%%%%
Hence, $\sum_{p=1}^\infty \P\{\Omega_{p,1}^c\} < \infty$.
Similarly, 
%%%%%here we use Lemma \ref{L:P:Talag} with $n=n_0$ (lower endpoint)%%%%%
\begin{align*}
    \P\{\Omega_{p,2}^c\} 
    \le \frac{1}{2^{-(1+d)2^{4p}}\lambda_N(I_1)} 2^{-(1+2d)2^{4p}}\lambda_N(I_1)
\end{align*}
and hence $\sum_{p=1}^\infty \P\{ \Omega_{p,2}^c\} < \infty$.
Consider the event
\begin{align}\label{Omega_p3}
    \Omega_{p,3} = \bigcap_{\ell = p}^\infty \mathcal{E}_\ell, \quad \text{where } 
    \mathcal{E}_\ell = \left\{ \sup_{t,s\in I: |t-s|\le \sqrt{N} 2^{-\ell}} |X(t)-X(s)| \le K_3 \sigma(2^{-\ell}) \sqrt{\ell} \right\}.
\end{align}
    By Lemma 3.1 of \cite{X96}, we can fix a constant $K_3>0$ such that $\P(\mathcal{E}_\ell^c) \le e^{-\ell}$
for all sufficiently large $\ell$. %%%%% This is obtainted by applying equation (3.1) from Lemma 3.1 in \cite{X96} with r=\sqrt{N}2^{\ell} and u=K_3\sigma(2^{\ell})\sqrt{\ell} %%%%%
It follows that $\sum_{p=1}^\infty \P\{\Omega_{p,3}^c\} < \infty$.
Let $\Omega_p = \Omega_{p,1} \cap \Omega_{p,2} \cap \Omega_{p,3}$. 
Then
\[
    \sum_{p=1}^\infty \P\{\Omega_p^c\} < \infty.
\]
By the Borel--Cantelli lemma, with probability 1, $\Omega_p$ occurs for all sufficiently large $p$.

For any ball $A$ in $\R^d$, denote its radius by $r_A$. 
Let $\mathscr{F}_{p,1}$ be the family of closed balls $A$ in $\R^d$ with 
radius {$R_{2p} \le r_A \le 4R_p$} such that 
\begin{align}\label{A:F1}
    \lambda_N\{t \in I_2 : X(t) \in A \} \ge c_1 r_A^d \Psi(r_A) \log\log\log(1/r_A),
\end{align}
where $c_1>0$ is a constant such that
\begin{align}\label{c_1}
    c_0pr^d\Psi(r) \ge c_1 (4r)^d \log\log\log(1/(4r))  \Psi(4r) \quad \text{for all $p \ge 1$, $R_{2p}\le r\le R_p$.}
\end{align}
Let $\mathscr{F}_{p,1}'$ and $\mathscr{F}_{p,1}''$ be the families of balls obtained by applying Lemma \ref{L:Vitali} to $\mathscr{F}_{p,1}$, that is, $\mathscr{F}_{p,1}'$ is a subfamily of $\mathscr{F}_{p,1}$ containing disjoint balls, and $\mathscr{F}_{p,1}'' = \{ 5A : A\in \mathscr{F}_{p,1}' \}$ covers $\mathscr{F}_{p,1}$.

Next, consider the family $\mathscr{F}_{p,2}$ of closed balls $A$ in $\R^d$ with radius {$R_{4p} \le r_A \le 4R_{2p}$} that are disjoint from the balls in $\mathscr{F}_{p,1}''$ and satisfy
\begin{align}\label{A:F2}
    \lambda_N\{ t \in I_1 : X(t) \in A \} \ge c_2 n_0 r_A^d \Psi(r_A),
\end{align}
{where $c_2>0$ is a constant such that
\begin{align}\label{c_2}
    c_0r^d\Psi(r) \ge c_2 (4r)^d \Psi(4r) \quad \text{for all $p \ge 1$, $R_{4p}\le r\le R_{2p}$.}
\end{align}}
Similarly, let $\mathscr{F}_{p,2}'$ and $\mathscr{F}_{p,2}''$ be the families obtained by applying Lemma \ref{L:Vitali} to $\mathscr{F}_{p,2}$.

By \eqref{A:F1} and the property that the balls in the subfamily $\mathscr{F}_{p,1}'$ of $\mathscr{F}_{p,1}$ are disjoint, we have
\begin{align}\label{sum:F_p1}
    \sum_{A \in \mathscr{F}_{p,1}'} r_A^d \Psi(r_A) \log\log\log(1/r_A)
    \le c_1^{-1} \lambda_N(I_2).
\end{align}
{Next, observe that if $t \in I_1$, $X(t) \in A$ and $A \in \mathscr{F}_{p,2}$, then $t \not\in U_p$. Otherwise there exists $r \in 
[R_{2p},R_p]$ such that $\lambda_N\{ s \in I_2 : X(s)\in B(X(t), 4r) \} \ge 
c_0 p\, r^d \Psi(r)$.
Let $\widetilde{A}:=B(X(t),4r)$ with $r_{\widetilde{A}}:=4r$. Since $r_{\widetilde{A}}\in [R_{2p},4R_{p}]$, then by \eqref{c_1} we obtain that 
$$\lambda_N\{ s \in I_2 : X(s)\in \widetilde{A} \}\ge c_1 r^d_{\widetilde{A}}  \log\log\log(1/r_{\widetilde{A}})\Psi(r_{\widetilde{A}}).$$ Then $\widetilde{A}$ belongs to $\mathscr{F}_{p,1}$ and is thus covered by balls of $\mathscr{F}_{p,1}''$, but  $X(t)\in A\cap \widetilde{A}$, which is a contradiction since the balls of $\mathscr{F}_{p,1}''$ and $\mathscr{F}_{p,2}$ are disjoint.} It follows that
\begin{align*}
    \bigcup_{A \in \mathscr{F}_{p,2}'} \{t \in I_1 : X(t) \in A \} \subset I_2 \setminus U_p.
\end{align*}
The preceding and \eqref{A:F2} imply that on $\Omega_{p,1}$,
\begin{align*}
    \sum_{A \in \mathscr{F}_{p,2}'} r_A^d \Psi(r_A)
    \le c_2^{-1}n_0^{-1} \lambda_N(I_2 \setminus U_p) 
    \le c_2^{-1}n_0^{-1} 2^{-2^p} \lambda_N(I_2),
\end{align*}
and since $\log\log\log(1/r_A) \le K_4 p$ for some constant $K_4$, we have
\begin{align}\label{sum:F_p2}
    \sum_{A \in \mathscr{F}_{p,2}'} r_A^d \Psi(r_A) \log\log\log(1/r_A) 
    \le K_4 c_2^{-1}n_0^{-1} \lambda_N(I_2)  p2^{-2^p}.
\end{align}
Consider the family $\mathscr{G}_p$ of balls defined by 
\begin{align}\label{G_p}
    \mathscr{G}_p = \{{13}A : A \in \mathscr{F}_{p,1}'\} \cup \{{5}A: A \in \mathscr{F}_{p,2}'\}.
\end{align}
For each $p \ge 1$, let $\ell_p$ be the smallest positive integer such that 
\begin{align}\label{r_p}
    r_p:=K_3 \sigma(2^{-\ell_p}) \sqrt{\ell_p} \le R_{4p} = 2^{-2^{2^{4p}}},
\end{align}
where $K_3$ is the constant in \eqref{Omega_p3}.
It follows that, for some constants $K_5>K_6>0$,
\begin{align}\label{E:ell}
    K_6 2^{2^{4p}} \le \ell_p \le K_5 2^{2^{4p}}.
\end{align}
Let $\mathscr{H}_p$ be the family of all dyadic cubes $Q$ of order $\ell_p$ in $I_{1/2}$ such that $X(Q)$ intersects $X(I_{1/2}) \setminus \bigcup_{B \in \mathscr{G}_p} B$
(and thus $\{ X(Q) : Q \in \mathscr{H}_p \}$ covers $X(I) \setminus \bigcup_{B \in \mathscr{G}_p} B$ ),
%%%%%%%%%%%%%%%%%%%%Justification%%%%%%%%%%%%%%%%%%%%%%%%%
%if $t\in I$ such that $X(t) \notin \bigcup_{B \in \mathscr{G}_p} B$, then necessarily $X(t)\in X(Q_t)\cap X(I_{1/2}) \setminus \bigcup_{B \in \mathscr{G}_p} B\neq \varnothing$ where $Q_t$ is the unique dyadic cube of order $\ell_p$ in $I_{1/2}$ containing $t$. Hence $Q_t\in \mathcal{H}_p$. So $X(I) \setminus \bigcup_{B \in \mathscr{G}_p} B$ is covered by $\bigcup_{Q\in \mathcal{H}_p}X(Q)$,
%%%%%%%%%%%%%%%%%%%%%%%%%%%%%%%%%%%%%%%%%%%%%%%%%%%%%%%%%
and let $t_Q$ denote the center of $Q$.
On $\Omega_{p,3}$, for every $Q \in \mathscr{H}_p$,
\begin{align}\label{mc:Q}
    \sup_{t,s \in Q} |X(t)-X(s)| \le K_3 \sigma(2^{-\ell_p}) \sqrt{\ell_p} = r_p \le R_{4p},
\end{align}
and thus $X(Q)$ can be covered by the closed ball $B(X(t_Q), r_p)$ in 
$\R^d$. 

\noindent
{\bf Claim:} Every cube $Q \in \mathscr{H}_p$ is contained in 
$I_1\setminus V_p$. 

\noindent
{\it Proof of the {\bf Claim}.}
Suppose towards a contradiction that there exists a point $t \in Q \cap V_p$.
Then by the definition of $V_p$ there exists $r \in [R_{4p}, R_{2p}]$ 
such that 
\[
    \lambda_N\{s \in I_1: X(s) \in B(X(t),4r) \} \ge c_0 n_0 r^d\Psi(r).
\]
Let $\widetilde{A}:=B(X(t),4r)$ with $r_{\widetilde{A}}:=4r$. Since $r_{\widetilde{A}}\in [R_{4p},4R_{2p}]$, then by \eqref{c_2} we have 
\[
    \lambda_N\{ s \in I_1 : X(s)\in \widetilde{A} \}\ge c_2 n_0 r^d_{\widetilde{A}}\Psi(r_{\widetilde{A}}).
\] 
Case 1: If $\widetilde{A}\cap \bigcup_{B \in \mathscr{F}_{p,1}''} B = \varnothing$, then $\widetilde{A}$ belongs to $\mathscr{F}_{p,2}$ and hence $\widetilde{A}\subset \bigcup_{A\in \mathscr{F}_{p,2}'} 5 A$.\\
Case 2: If $\widetilde{A}\cap \bigcup_{B \in \mathscr{F}_{p,1}''} B \neq \varnothing$, then $\widetilde{A}\cap 5A\neq \varnothing$ for some $A\in \mathscr{F}'_{p,1}$, so we can find some $x^* \in \widetilde{A}\cap 5A$.
Since $r\leq R_{2p}\leq r_A$, for every $x\in \widetilde{A}$, we have 
\[
    |x-x_A|\le |x-X(t)|+|X(t)-x^*|+|x^*-x_A|\le 8r+5r_A\le 13r_A.
\]
This shows that $\widetilde{A}\subset \bigcup_{A\in \mathscr{F}_{p,1}'} 13 A$.
% and hence $\widetilde{A}\subset \bigcup_{A\in \mathscr{F}_{p,1}'} 13 A$, where we have used the fact $r\leq R_{2p}\leq r_A$. %%%%%%%%%If $x\in \widetilde{A}$, then $|x-x_A|\le |x-x'|+|x'-x_A|\le 8r+5r_A\le 13r_A$, where x'\in \widetilde{A}\cap 5A.%%%%%%%%

Combining both cases and recalling the definition of $\mathscr{G}_p$ in \eqref{G_p}, we have $\widetilde{A} \subset \bigcup_{B \in \mathscr{G}_p} B$.
But then \eqref{mc:Q} and $r\in [R_{4p},4R_{2p}]$ imply that $X(Q) \subset B(X(t),4r)=\widetilde{A} \subset \bigcup_{B \in \mathscr{G}_p} B$,
which is a contradiction to the definition of $\mathscr{H}_p$.
Hence, every cube $Q \in \mathscr{H}_p$ must be contained in $I_1\setminus V_p$.
This proves the {\bf Claim}. \qed

From the {\bf Claim}, it follows that 
\begin{align}\label{card:H_p}
    \# \mathscr{H}_p \le C2^{N\ell_p} 2^{-(1+d)2^{4p}} \quad \text{on the event $\Omega_{p,2}$.}
\end{align}
Now, $\mathscr{C}_p:=\mathscr{G}_{p} \cup \{ B(X(t_Q), r_p) : Q \in \mathscr{H}_p \}$ is a family of balls in $\R^d$ with radius at most $13R_p$ that cover $X(I)$.
Recall the function $\phi$ defined in \eqref{Eq:phi_def}.
Recall that on an event of probability 1, $\Omega_p$ occurs for all large $p$.
On this event, it follows from \eqref{sum:F_p1}, \eqref{sum:F_p2}, \eqref{r_p} and \eqref{card:H_p} that for all large $p$,
\begin{align}\label{E:phi-sum}
    \sum_{A \in \mathscr{C}_p} \phi(2r_A) 
    &= \sum_{A \in \mathscr{F}_{p,1}'} \phi(26r_A)
    + \sum_{A \in \mathscr{F}_{p,2}'} \phi(10r_A) + \sum_{Q \in \mathscr{H}_p} \phi(2r_p)\nonumber\\
    & \lesssim \sum_{A \in \mathscr{F}_{p,1}'} \phi(r_A) + \sum_{A \in \mathscr{F}_{p,2}'} \phi(r_A) + \# \mathscr{H}_p \cdot r_p^d \cdot (\log(1/r_p))^{1-\gamma d} \cdot \log\log\log(1/r_p)\nonumber\\
    & \lesssim \lambda_N(I_2) + \lambda_N(I_2) p 2^{-2^p} + 2^{N\ell_p} 2^{-(1+d)2^{4p}} \cdot 2^{-Hd\ell_p} \ell_p^{\gamma d} \ell_p^{d/2} \cdot \ell_p^{1-\gamma d} \cdot \log\log \ell_p \nonumber\\
    &\lesssim \lambda_N(I_2) (1+o(1))+ \ell_p^{-d/2} \log\log \ell_p,
\end{align}
where we have used $N=Hd$ and \eqref{E:ell} to obtain the last inequality. Therefore, with probability 1, for all large $p$,
\begin{align*}
    \sum_{A \in \mathscr{C}_p} \phi(2r_A) 
    & \lesssim \lambda_N(I_2) (1+o(1)) + o(1).
\end{align*}
This shows that $\mathcal{H}^\phi(X(I))<\infty$ a.s.
Since $r^d/\phi(r) = o(1)$ as $r\to0^+$, $X(I)$ has Lebesgue measure 0 a.s.
The proof of Theorem \ref{T:Haus:meas} is complete.
\end{proof}
% \begin{remark}

% \end{remark}

\section{Proof of Theorem \ref{T:polar}}
\label{S:pf:T3}

We start with two auxiliary results, which will be used to prove part (i) and part (ii) of Theorem \ref{T:polar}, respectively.

\begin{lemma}\label{lem:PZ}
    Let $(\mu_n)_{n \ge 1}$ be a sequence of random positive Borel measures on a compact set $I \subset \R^N$.
    Suppose there exist two constants $C_1, C_2 \in (0,\infty)$ such that
    \begin{align}\label{mu:moment:bd}
        \E[\mu_n(I)] \ge C_1 \quad \text{and}\quad
        \E[(\mu_n(I))^2] \le C_2 \quad \text{for all $n \ge 1$.}
    \end{align}
    Then, on an event $\Omega_0$ of probability $\P(\Omega_0)\ge {C_1^2}/{8C_2}$,\ $(\mu_n)_{n \ge 1}$ has a subsequence that converges weakly to a random measure $\mu$ on $I$ such that $\mu(I) \ge C_1/2>0$ on $\Omega_0$.
\end{lemma}

\begin{proof}
This lemma is a folklore and has been utilized by several authors (cf., e.g., \cite{Te86,BLX09}). For easy reference, we provide a proof here.
    
By the Paley--Zygmund inequality \eqref{E:PZ} and \eqref{mu:moment:bd}, for every $n \ge 1$,
    \begin{align*}
        \P\left\{ \mu_n(I) \ge \frac{C_1}{2} \right\} \ge \P\left\{ \mu_n(I) \ge \frac{\E[\mu_n(I)]}{2} \right\}
        \ge \frac{\E[\mu_n(I)]^2}{4\E[(\mu_n(I))^2]} \ge \frac{C_1^2}{4C_2}.
    \end{align*}
    It follows from the preceding and Markov's inequality that for any $M>C_1/2$,
    \begin{align*}
        \P\left\{ \frac{C_1}{2} \le \mu_n(I) \le M \right\} 
        \ge  \P\left\{ \mu_n(I) \ge \frac{C_1}{2} \right\} - \P\left\{ \mu_n(I) > M \right\}
        \ge \frac{C_1^2}{4C_2} - \frac{C_2}{M^2}.
    \end{align*}
    Hence, we may choose $M$ large enough so that
    \begin{align*}
        \P\left\{ \frac{C_1}{2} \le \mu_n(I) \le M \right\} 
        \ge \frac{C_1^2}{8C_2}.
    \end{align*}
    Let
    \begin{align*}
        \Omega_0 = \left\{ \frac{C_1}{2} \le \mu_n(I) \le M \text{ infinitely often} \right\}
    \end{align*}
    Then
    \begin{align*}
        \P(\Omega_0) &= \P\left( \limsup_{n\to\infty} \left\{ \frac{C_1}{2} \le \mu_n(I) \le M \right\} \right) \\
        &\ge \limsup_{n \to \infty} \P\left\{ \frac{C_1}{2} \le \mu_n(I) \le M \right\} 
        \ge \frac{C_1^2}{8C_2} > 0.
    \end{align*}
    On the event $\Omega_0$, $(\mu_n)_{n \ge 1}$ is a sequence of measures whose total variation norms are bounded by $M$ and are tight because they are supported in the compact set $I$.
    Hence, by Prohorov's theorem, $(\mu_n)_{n \ge 1}$ has a subsequence that converges weakly to a measure $\mu$.
    In particular, the weak convergence implies that, on $\Omega_0$,
    \begin{align*}
        \mu(I) = \lim_{n\to\infty} \mu_n(I) \ge C_1/2.
    \end{align*}
    This completes the proof.
\end{proof}

\begin{lemma}\label{lem:X3}
    Let Assumptions \ref{a:X}, \ref{a:var}, and \ref{a:SLND} hold with $\sigma$ given by \eqref{sigma:gamma} with $N=Hd$ and $0<\gamma \le 1/d$.
    Fix a compact interval $I \subset \R^N \setminus \{0\}$ and $t_0 \in I$.
    Consider the Gaussian fields $X^{(1)}$ and $X^{(2)}$ defined by \eqref{X2:X1}.
    Assume that $\alpha$ given by \eqref{X2} satisfies
    \begin{align}\label{alpha}
        1/2 \le \alpha(t) \le 3/2 \quad \text{for all $t \in I$.}
    \end{align}
    Fix $z \in \R^d$ and consider the Gaussian field $X^{(3)}$ defined by \eqref{X3}, i.e.,
    \[
        X^{(3)}(t) = \frac{1}{\alpha(t)}(z-X^{(1)}(t)).
    \]
    Then $X^{(3)}(I)$ has Lebesgue measure 0 a.s.
\end{lemma}

\begin{proof}
Let $R_p$, $\beta_0\vee 1/\gamma<\beta<1/H$, $c_0>0$ and $n_0,p_0 \in \N^+$ be 
given by Lemma \ref{L:P:Talag}, where $\gamma=2H\wedge 1$.
For any $p \in \N^+$, define the random sets $U_p'$, $V_p'$ by
\begin{align*}
    &U_p' = \left\{ t \in I_2 : \exists\, r \in [R_{2p},R_p], \lambda_N\{ s \in B(t,r^{\beta}) : |X^{(3)}(t)-X^{(3)}(s)| \le c_1 r \} \ge c_0 p\, r^d \Psi(r) \right\},\\
    &V_p' = \left\{ t \in I_1 : \exists\, r \in [R_{4p},R_{2p}], \lambda_N\{ s \in B(t,r^{\beta}) : |X^{(3)}(t)-X^{(3)}(s)| \le c_1r \} \ge c_0 n_0 r^d \Psi(r) \right\},
\end{align*}
where $c_1>0$ is a constant to be specified later (see \eqref{E:X3-X3-Up'} below). 
Recall the events $\Omega_{p,1}$ and $\Omega_{p,2}$ defined in \eqref{Omegap1:Omegap2}.
Fix $\eta>0$ and define the events $\Omega_{p,0}$, $\Omega'_{p,1}$ and $\Omega'_{p,2}$ by
\begin{align*}
    &\Omega_{p,0} = \{ \textstyle\sup_{t \in I}|X(t)| \le 2^{\eta p} \},\\
    &\Omega'_{p,1} = \{ \lambda_N(U'_p) \ge (1-2^{-2^p}) \lambda_N(I_2) \},\\
    &\Omega'_{p,2} = \{ \lambda_N(V'_p) \ge (1-2^{-(1+d)2^{4p}}) \lambda_N(I_1) \}.
\end{align*}
Then $\Omega_{p,0}\cap \Omega_{p,1}\subset \Omega'_{p,1}$ for $p$ sufficiently large.  Indeed, if $t \in U_p$ and $\Omega_{p,0}\cap \Omega_{p,1}$ occurs for $p$ large, then there exists $r \in [R_{2p},R_p]$ such that 
$$\lambda_N\{ s \in B(t,r^{\beta}) : |X(t)-X(s)| \le 4r \} \ge c_0 p\, r^d \,\Psi(r).$$ 
If $s\in B(t,r^{\beta})$ such that $|X(t)-X(s)|\leq 4r$, then by \eqref{E:alpha-alpha}, \eqref{E:X2-X2}, \eqref{X3-X3}, and \eqref{alpha}, we have
\begin{align}\begin{split}\label{E:X3-X3-Up'}
    |X^{(3)}(t)-X^{(3)}(s)|&\leq 4K_0(|z|+2^{\eta p})r^{\beta \gamma}+ 2(4r+K_0\,r^{\beta\gamma}2^{\eta p})\\
    &\leq K_0\big({4|z|}+6\big)\,r+{8}\,r=:c_1\,r. 
\end{split}\end{align}
This shows that $U_p \subset U'_p$, hence verifying that $\Omega_{p,0}\cap \Omega_{p,1}\subset \Omega'_{p,1}$ for $p$ large. Similarly, we have $\Omega_{p,0}\cap \Omega_{p,2}\subset \Omega'_{p,2}$ for $p$ sufficiently large. Next, recall the events $\Omega_{p,3}$ and  $\mathcal{E}_\ell$ in \eqref{Omega_p3}, and consider the event
\begin{align*}
% \label{Omega'_p3}
    \Omega'_{p,3} = \bigcap_{\ell = p}^\infty \mathcal{E}'_\ell, \quad \text{where } 
    \mathcal{E}'_\ell = \left\{ \sup_{t,s\in I: |t-s|\le \sqrt{N} 2^{-\ell}} |X^{(3)}(t)-X^{(3)}(s)| \le K_3' \sigma(2^{-\ell}) \sqrt{\ell} \right\},
\end{align*}
where the constant $K_3'>0$ can be chosen so that $\Omega_{p,0}\, \cap\,\mathcal{E}_{\ell}\subset \mathcal{E}'_{\ell}$ for all $\ell\ge p$. Then $\Omega_{p,0}\,\cap \Omega_{p,3}\subset \Omega'_{p,3}$ for $p$ large. Let $\Omega'_p = \Omega'_{p,1} \cap \Omega'_{p,2} \cap \Omega'_{p,3}$. 
Then
\[
    \sum_{p=1}^\infty \P\{(\Omega'_p)^c\} < \infty.
\]
By the Borel--Cantelli lemma, with probability 1, $\Omega'_p$ occurs for all sufficiently large $p$.

Following the same steps of \eqref{A:F1}, \dots, \eqref{E:phi-sum} in the proof of Theorem \ref{T:Haus:meas}, we obtain a family of balls in $\R^d$, $\mathscr{C}_p:= \{ B(x_j, r_j) : j\in J\}$ with radius at most $c_2R_p$ (here $c_2$ is a constant depending on $c_1$ defined above) that cover $X^{(3)}(I)$ such that for all large $p$,
\begin{align*}
\sum_{j\in J}\phi(2r_i)\lesssim \lambda_N(I_2)(1+o(1))+ o(1).
\end{align*}
This shows that $\mathcal{H}^{\phi}(X^{(3)}(I))<\infty$ a.s. Since $r^d/\phi(r)=o(1)$ as $r\to0^+$, $X^{(3)}(I)$ has Lebesgue measure 0 a.s. The proof is complete.
\end{proof}

\begin{proof}[Proof of Theorem \ref{T:polar}]
(i). Suppose \eqref{E:int:cond} fails. Let $I = [\delta_0/2,\delta_0]^N$.
Then, by using polar coordinates, we see that
\begin{align}\label{iint:finite}
    \int_I \int_I \frac{dt\, ds}{\sigma^d(|t-s|)} < \infty.
\end{align}
% It is known \cite{GH80} that the above condition implies that a square-integrable local time $L(z,I)$ exists and $L(z,\cdot \cap I)$ defines a non-trivial measure on $X^{-1}\{z\} \cap I$.
% Then $\{z\}$ is not polar for $\{X(t), t \in \R^N \setminus\{0\} \}$.
We will prove that $\{z\}$ is not polar by constructing a measure that is supported on the level set $X^{-1}\{z\} \cap I$ and is non-trivial with positive probability.
For each $n \in \N$ and for each Borel subset $A$ of $I$, define
\begin{align}\begin{split}\label{mu_n}
	\mu_n(A) &:= \int_A (2\pi n)^{d/2} \exp\left(- \frac{n|X(t)-z|^2}{2}\right) dt\\
	&= \int_A \int_{\R^d} \exp\left(-i\xi \cdot (X(t)-z)-\frac{|\xi|^2}{2n}\right) d\xi \, dt,
\end{split}\end{align}
where the last identity can be verified easily using the characteristic
function of a normal distribution. Following \cite[p.185-186]{X09} (see
also \cite[p.13-14]{BLX09}), we deduce that
\begin{align*}
	\E[\mu_n(I)] &= \int_I \int_{\R^d} e^{-i\xi\cdot z} \exp\left(-\frac{|\xi|^2}{2n}\right) \E[e^{-i\xi \cdot X(t)}]\,d\xi\, dt\\
	&= \int_I \int_{\R^d} e^{-i\xi\cdot z} \exp\left(-\frac{(n^{-1} + d^2(t,0))|\xi|^2}{2}\right) d\xi dt\\
	&= \int_I \left( \frac{2\pi}{n^{-1} + d^2(t,0)} \right)^{d/2} \exp\left( - \frac{|z|^2}{2(n^{-1}+d^2(t,0))} \right) dt\\
	&\ge \int_I \left( \frac{2\pi}{1 + c_1\sigma^2(t)} \right)^{d/2} \exp\left( - \frac{|z|^2}{2c_2 \sigma^2(t)} \right) dt,
\end{align*}
where the last line follows from \eqref{variogram} and $d^2(t,0) \ge 
c_2 \sigma^2(t)$ which follows from \eqref{SLND}. Recall from 
\eqref{sigma} that $\sigma$ is continuous on $I$ and takes the form
\begin{align}\label{sigma:LB}
	\sigma(|t|) = |t|^H L(|t|) \ge C_0:=\inf_{s \in I} |s|^H L(|s|) > 0 \quad \text{for all $t \in I=[\delta_0/2,\delta_0]^N$,}
\end{align}
so we can find a positive constant $C_1>0$ such that
\begin{align}\label{E(mu)}
	\E[\mu_n(I)] \ge C_1 \quad \text{for all $n \in \N$.}
\end{align}
Let $I_{2d}$ be the $2d \times 2d$ identity matrix, let $\mathrm{Cov}(X(t),X(s))$ be the $2d\times 2d$ covariance matrix of the Gaussian vector $(X(t),X(s))$, let $\Gamma_n(t,s) = \frac1n I_{2d} + \mathrm{Cov}(X(t),X(s))$, and let $(\xi,\eta)'$ be the transpose of the row vector $(\xi,\eta)$.
Again, following \cite[p.185-186]{X09} (see also \cite[p.13-14]{BLX09}), we also have
\begin{align*}
	\E[(\mu_n(I))^2]
	&= \int_I \int_I \int_{\R^d} \int_{\R^d} e^{-i(\xi,\eta)\cdot (z,z)} \exp\left( - \frac{1}{2} (\xi,\eta) \Gamma_n(t,s) (\xi,\eta)' \right) d\xi \, d\eta \, dt \, ds\\
	&\le \int_I \int_I \int_{\R^d} \int_{\R^d}  e^{ - \frac{1}{2} (\xi,\eta) \mathrm{Cov}(X(t),X(s)) (\xi,\eta)'} d\xi \, d\eta \, dt \, ds\\
	&= \int_I \int_I \frac{(2\pi)^d}{[\det \mathrm{Cov}(X(t),X(s))]^{1/2}}\, dt \, ds\\
	&= \int_I \int_I \frac{(2\pi)^d}{[\mathrm{Var}(X_1(s)) \mathrm{Var}(X_1(t)|X_1(s))]^{d/2}}\, dt \, ds\\
	& \le \frac{(2\pi)^d}{(c_2 C_0^2)^{d/2}\, c_2^{d/2}} \int_I \int_I \frac{1}{\sigma^d(|t-s|)}\, dt\, ds,
\end{align*}
where we have used \eqref{sigma:LB} and \eqref{SLND} to obtain the last line.
By \eqref{iint:finite}, there is a constant $C_2<\infty$ such that
\begin{align}\label{E(mu^2)}
	\E[(\mu_n(I))^2] \le C_2 \quad \text{for all $n \in \N$.}
\end{align}
Thanks to \eqref{E(mu)} and \eqref{E(mu^2)}, we can apply Lemma \ref{lem:PZ} 
to find that there is an event of positive probability on which $(\mu_n)_{n \ge 1}$ has a subsequence that converges weakly to a random measure $\mu$ on $I$, such that $\mu(I) \ge C_1/2>0$ on $\Omega_0$.
This shows that $\mu$ is a non-trivial measure on $I$ with positive probability.
As the weak limit of a subsequence of the measures $(\mu_{n})_{n \ge 1}$, we can observe from the definition \eqref{mu_n} that $\mu$ is supported on the level set $X^{-1}\{z\} \cap I$.
Therefore, this proves that $X^{-1}\{z\} \cap I \ne \varnothing$ with positive probability. 
%\end{proof}

(ii). 
% Assume \eqref{cond:f(sigma*)} and \eqref{cond:f(x)L^d(x) bound}.
% Note that \eqref{cond:f(sigma*)} implies \eqref{intsigma}, which in turn implies $N\le Hd$.
Suppose $\sigma$ is given by \eqref{sigma:gamma}.
Note that, in this case, \eqref{E:int:cond} implies $N \le Hd$ (see \eqref{int:cond:under:gamma}).

Case 1: $N< Hd$. In this case, \eqref{suff-cond} holds.
Therefore, Theorem \ref{T:polar:suff} implies that points are polar for $X$.

Case 2: $N=Hd$.
Fix $z \in \R^d$.
It suffices to show that for any fixed $t_0 \in \R^N \setminus \{0\}$, there is a closed interval $I \subset \R^d \setminus \{0\}$ centered at $t_0$ with diameter $\rho_0>0$ such that
\[
\P\{ \exists\, t \in I, X(t) = z \} = 0.
\]
Consider the Gaussian random fields $X^{(1)}$, $X^{(2)}$, and $X^{(3)}$ defined in \eqref{X2:X1} and \eqref{X3}.
Under the current assumptions, as in \eqref{alpha:UB:LB}, we may choose $\rho_0 \in (0,\delta_0)$ such that
$1/2 \le \alpha(t) \le 3/2$ for all $t_0 \in I$, where $I \subset \R^d \setminus \{0\}$ is the closed interval centered at $t_0$ with diameter $\rho_0$.
Lemma \ref{lem:X3} shows that $X^{(3)}(I)$ has Lebesgue measure 0 a.s.
By Fubini's theorem,
\begin{align*}
\int_{\R^d} \P\{ \exists\, t \in I, X^{(3)}(t) = y \}\, dy 
= \E[\lambda_d(X^{(3)}(I))] = 0.
\end{align*}
This implies that
\begin{align}\label{P0aey}
\P\{ \exists\, t \in I, X^{(3)}(t) = y \} \text{ for almost every } y \in \R^d.
\end{align}
Let $f_0(y)$ denote the probability density function of $X(t_0)$.
Note that $X(t) = z$ if and only if $X^{(3)}(t) = X(t_0)$. 
Also, $X^{(1)}$, and hence $X^{(3)}$, is independent of $X(t_0)$. 
It follows that
\begin{align*}
\P\{ \exists\, t \in I, X(t) = z \} &= \P\{ \exists\, t \in I, X^{(3)}(t) = X(t_0) \}\\
& = \int_{\R^d} \P\{ \exists\, t \in I, X^{(3)}(t) = y \}\, f_0(y)\, dy.
\end{align*}
Using \eqref{P0aey}, we conclude that $\P\{ \exists\, t \in I, X(t) = z \} = 0$.
\end{proof}

\section*{Acknowledgements}
C.Y. Lee was supported in part by the Shenzhen Peacock grant 2025TC0013. Y. Xiao was supported in part by the NSF grant DMS-2153846.


\begin{thebibliography}{99}

\bibitem{A55}
Anderson, T. W. (1955). The integral of a symmetric unimodal function over a symmetric convex set and some probability inequalities. Proc. Amer. Math. Soc. 6, 170--176.

\bibitem{Anh99}
Anh, V. V., Angulo, J. M.  and Ruiz-Medina, M. D.  (1999). Possible 
long-range dependence in fractional random fields.  J. Statist. Plann.
Inference, 80, 95--110.

\bibitem{Berman69}
Berman, S. M. (1969). Local times and sample function properties of 
stationary Gaussian processes. Trans. Amer. Math. Soc. 137, 277--299.

\bibitem{BLX09}
Bierm\'e, H., Lacaux, C. and Xiao, Y. (2009). Hitting probabilities
and the Hausdorff dimension of the inverse images of anisotropic 
Gaussian random fields, Bull. Lond. Math. Soc. 41, no. 2, 253--273.

\bibitem{BGT89}
Bingham, N. H., Goldie, C. M. and Teugels, J. L. (1989). Regular variation. Encyclopedia Math. Appl., 27, Cambridge University Press, Cambridge.

% \bibitem{Borell}
% Borell, C. (1975), The Brunn-Minkowski inequality in Gauss space, Invent. Math. 30, no. 2, 207--216.
\bibitem{DKN07}
Dalang, R. C., Khoshnevisan, D. and Nualart, E. (2007). Hitting 
probabilities for systems of non-linear stochastic heat equations with additive noise. ALEA Lat. Am. J. Probab. Math.
Stat. 3, 231--271.

\bibitem{DKN09}
Dalang, R. C., Khoshnevisan, D. and Nualart, E. (2009). Hitting probabilities for systems for non-linear stochastic heat equations with multiplicative noise. Probab. Theory Related Fields 144, no. 3-4, 371--427. 

\bibitem{DKN13}
Dalang, R. C., Khoshnevisan, D. and Nualart, E. (2013). Hitting probabilities for systems of non-linear stochastic heat equations in spatial dimension $k \ge 1$. Stoch. Partial Differ. Equ. Anal. Comput. 1, no. 1, 94--151,

\bibitem{DMX17}
Dalang, R. C., Mueller, C. and Xiao, Y. (2017). Polarity 
of points for Gaussian random fields. Ann. Probab. 45, no.
6B, 4700--4751. 

\bibitem{DN04}
Dalang, R. C. and Nualart, E. (2004). Potential theory for 
hyperbolic SPDEs. Ann. Probab. 32, no. 3A, 2099--2148,

\bibitem{DS10}
Dalang, R. C. and Sanz-Sol\'e, M. (2010). Criteria for 
hitting probabilities with applications to systems of stochastic wave equations. Bernoulli 16, no. 4, 1343--1368.

\bibitem{DS15}
Dalang, R. C. and Sanz-Sol\'e, M. (2015). Hitting 
probabilities for nonlinear systems of stochastic waves. 
Mem. Amer. Math. Soc. 237, no. 1120,

\bibitem{Dudley}
Dudley, R. M. (1967). The sizes of compact subsets of Hilbert space and continuity of Gaussian processes. J. Functional Analysis 1, 290--330.

\bibitem{EH25} Erraoui, M. and Hakiki, Y. (2025). Fractional Brownian motion with deterministic drift: how critical is drift regularity in hitting probabilities. Math. Proc. Camb. Philos. Soc. 178, no. 1, 103--132.  

\bibitem{Falconer}
Falconer, K. (2013). Fractal geometry: mathematical foundations and applications. John Wiley \& Sons.

% \bibitem{Fernique}
% Fernique, X. (1970), Int\'egrabilit\'e des vecteurs gaussiens, C. R. Acad. Sci. Paris S\'er. A-B 270, A1698--A1699.

\bibitem{FKN11}
Foondun, M., Khoshnevisan, D., and Nualart, E. (2011), 
A local-time correspondence for stochastic partial differential equations. Trans. Amer. Math. Soc., 363, no. 5, 2481--2515.

\bibitem{GH80}
Geman, D. and Horowitz, J. (1980). Occupation densities. Ann. Probab. 8, no. 1, 1--67. 

\bibitem{H86}
Hawkes, J. (1986). Local times as stationary processes. In: From Local Times to Global Geometry. Elworthy, K.D. (Ed.), pp. 111–120, Pitman Research Notes in Mathematics, Vol. 150. Longman, Chicago.

\bibitem{HSWX}
Herrell, R., Song, R., Wu, D. and Xiao, Y. (2020). Sharp 
space-time regularity of the solution to a stochastic heat 
equation driven by a fractional-colored noise. Stoch. Anal.
Appl. 38, 747--768.

\bibitem{HS}
Hinojosa-Calleja, A. and Sanz-Sol\'e, M. (2021). Anisotropic Gaussian random fields: criteria for hitting probabilities and applications. Stoch. Partial Differ. Equ. Anal. Comput. 9, no. 4, 984--1030.

\bibitem{K85}
Kahane, J.-P. (1985). Some Random Series of Functions. 2nd
edition, Cambridge University Press.

\bibitem{Kesten69}
Kesten, H. (1969). Hitting probabilities of single points for processes with stationary independent increments. Amer. Math. Soc., Providence RI.

\bibitem{KS99}
Khoshnevisan, D. and Shi, Z. (1999). Brownian sheet and capacity. Ann. Probab. 27, no. 3, 1135--1159.

\bibitem{KXZa}
Khoshnevisan, D., Xiao, Y. and  Zhong, Y. (2003a). Measuring the range of an 
additive L\'evy process. Ann. Probab. 31, 1097--1141.

\bibitem{KXZb}
Khoshnevisan, D., Xiao, Y. and  Zhong, Y. (2003b).  Local times of additive L\'evy processes. Stoch. Process. Appl. 104, 193--216.

\bibitem{LM}
Lata{\l}a, R. and Matlak, D. (2017). Royen’s proof of the Gaussian correlation inequality, Geometric Aspects of Functional Analysis, Lecture Notes in Math., vol. 2169, Springer,
Cham, pp. 265--275.

\bibitem{LSXY23}
Lee, C. Y., Song, J., Xiao, Y. and Yuan, Y.  (2023).
Hitting probabilities of Gaussian random fields and 
collision of eigenvalues of random matrices. Trans. Amer. 
Math. Soc. 376, no. 6, 4273--4299. 

\bibitem{LX26}
Lee, C. Y. and Xiao, Y. (2026). Hitting probability, thermal capacity, and 
Hausdorff dimension results for the Brownian sheet. Electron. J. Probab. 31, 
no. 26, 1--31.

\bibitem{MarcusRosen06}
Marcus, M. B. and Rosen, J. (2006). Markov Processes, 
Gaussian Processes, and Local Times. Cambridge University 
Press, Cambridge.

\bibitem{Mattila}
Mattila, P. (1995). Geometry of sets and measures in Euclidean spaces. Fractals and rectifiability.
Cambridge Stud. Adv. Math., 44,  Cambridge University Press.

\bibitem{Pitt78}
Pitt, L. D. (1978). Local times for Gaussian random fields.  Indiana Univ. Math.  J., 27, 309--330.

\bibitem{Pitman68}
Pitman, E. J. G.  (1968). On the behavior of the characteristic function of a probability distribution in the neighbourhood of the origin. J. Australian Math. Soc. Series A, 8, 422--443.

\bibitem{Royen}
Royen, T. (2014). A simple proof of the Gaussian correlation conjecture extended to some multivariate gamma distributions, Far East J. Theor. Stat. 48, no. 2, 139--145.

% \bibitem{S68}
% Z. \v{S}id\'ak, On multivariate normal probabilities of rectangles: Their dependence on correlations. {\it Ann. Math. Statist.} {\bf 39} (1968), 1425--1434.

\bibitem{T95}
Talagrand, M. (1995). Hausdorff measure of trajectories of multiparameter fractional Brownian motion. Ann. Probab. 23, no. 2, 767--775.

\bibitem{T98}
Talagrand, M. (1998). Multiple points of trajectories of multiparameter fractional Brownian motion. Probab. Theory Relat. Fields 112, 545--563.

\bibitem{Te86}
Testard, F. (1986). Polarit\'e, points multiples et g\'eom\'etrie de certain processus gaussiens.  Publ. du Laboratoire de Statistique et
Probabilit\'es de l'U.P.S.,  Toulouse, 01 - 86.


\bibitem{X96}
Xiao, Y. (1996). Hausdorff measure of the sample paths of Gaussian random fields. Osaka J. Math.  33, no. 4, 895--913.

\bibitem{X97}
Xiao, Y. (1997). H\"older conditions for the local times and the Hausdorff measure of the level sets of Gaussian random fields. Probab. Theory Related Fields 109, no. 1, 129--157. 

\bibitem{X99}
Xiao, Y. (1999). Hitting probabilities and polar sets for
fractional Brownian motion. Stochastics
Stochastics Rep. 66, no. 1-2, 121--151.

\bibitem{X07}
Xiao, Y. (2007). Strong local nondeterminism and the sample path properties of Gaussian random fields. In: Asymptotic Theory in Probability and Statistics with Applications (Tze Leung Lai, Qiman Shao, Lianfen Qian, editors), pp. 136--176, Higher Education Press, Beijing.

\bibitem{X09}
Xiao, Y. (2009). Sample path properties of anisotropic Gaussian random fields.
In: A minicourse on stochastic partial differential equations. pp. 145--212. Lecture Notes in Math., vol. 1962, Springer, Berlin.

\bibitem{Yaglom1957}
Yaglom, A. M. (1957), Certain types of random fields in n-dimensional space similar to stationary stochastic processes.
Teor. Veroyatnost. i Primenen. 2, 292--338.

\bibitem{Yaglom1987}
Yaglom, A. M. (1987), Correlation theory of stationary and related random functions. Vol. I. Basic results. Springer Ser. Statist. Springer-Verlag, New York.

\end{thebibliography}
\end{document}